\documentclass[11pt,a4paper]{article}
\usepackage[utf8]{inputenc}
\usepackage{amsmath}
\usepackage{amsfonts}
\usepackage{amssymb}
\usepackage{amsthm}
\usepackage{mathrsfs}
\usepackage{mathtools}
\usepackage{enumerate}
\usepackage{hyperref}
\usepackage{tikz-cd}
\usepackage{stmaryrd}
\usepackage{pgf}

\usepackage[all]{xy}
\tikzset{node distance=2cm, auto}
\usepackage{enumerate}
\usepackage{color}

\allowdisplaybreaks

\input xy
\xyoption{matrix}\xyoption{arrow}
\def\edge{\ar@{-}}
\def\dedge{\ar@{.}}

\newtheorem{thm}{Theorem}[section]
\newtheorem{pro}[thm]{Proposition}
\newtheorem{lem}[thm]{Lemma}
\newtheorem{cor}[thm]{Corollary}
\newtheorem{hypo}{Hypothesis}

\theoremstyle{definition}

\newtheorem{defn}[thm]{Definition}
\newtheorem{exa}[thm]{Example}
\newtheorem{rem}[thm]{Remark}

\def\k{{\mathbb K}}

\newcommand{\der}{\operatorname{\mathsf{Der}}}
\newcommand{\aut}{\operatorname{\mathsf{Aut}}}
\newcommand{\innder}{\operatorname{\mathsf{InnDer}}}
\newcommand{\hh}{\operatorname{\mathsf{HH^1}}}
\newcommand{\ad}{\operatorname{\mathsf{ad}}}
\newcommand{\wt}{\operatorname{\mathsf{wt}}}
\newcommand{\fract}{\operatorname{\mathsf{Fract}}}
\newcommand{\q}{\mathrm{\bf{q}}}
\renewcommand\deg{\operatorname{\mathsf{deg}}}
\newcommand\rk{\operatorname{\mathsf{rk}}}

\newcommand{\Hom}{\operatorname{\mathsf{Hom}}}
\newcommand{\ZC}{\operatorname{\mathsf{Z}}}
\newcommand{\NS}{\operatorname{\mathsf{N}}}
\newcommand{\supp}{\operatorname{\mathsf{supp}}}

\def\ch{{\mathcal H}}

\newcommand{\Z}{\mathbb{Z}}

\title{Derivations and Hochschild cohomology of quantum nilpotent algebras} 

\usepackage[affil-it]{authblk}

\author[1]{St\'{e}phane Launois}
\affil[1]{Université de Caen Normandie, CNRS UMR 6139 LMNO, 14032 Caen, France.}

\author[2]{Samuel A.\ Lopes\thanks{Partially supported by CMUP -- Centro de Matem\'atica da Universidade do Porto, member of LASI, which is financed by national funds through FCT -- Funda\c c\~ao para a Ci\^encia e a Tecnologia, I.P., under the project with reference UID/00144.}}

\affil[2]{CMUP, Departamento de Matem\'atica, Faculdade de Ci\^encias, Universidade do Porto, Rua do Campo Alegre s/n, 4169--007 Porto, Portugal.}

\author[3]{Isaac Oppong}
\affil[3]{School of Computing and Mathematical Sciences,
University of Greenwich, 
Old Royal Naval College, Park Row,
London SE10 9LS, UK.}

\newcounter{marg}[section]

\usepackage[normalem]{ulem}

\begin{document}
\maketitle
\begin{abstract}
We compute the derivations of Quantum Nilpotent Algebras under a technical (but necessary) assumption on the center. As a consequence, we give an explicit description of the first Hochschild cohomology group of $U_q^+(\mathfrak{g})$, the positive part of the quantized enveloping algebra of a finite-dimensional complex simple Lie algebra $\mathfrak{g}$. Our results are obtained leveraging an initial cluster constructed by Goodearl and Yakimov. 
\end{abstract}


\section{Introduction}

Quantum nilpotent algebras (QNAs, for short), also known as CGL extensions (after Cauchon--Goodearl--Letzter), have been widely studied since their introduction in \cite{stl}, particularly in \cite{yak} and \cite{yak1}, where the authors construct quantum cluster algebra structures on QNAs satisfying a few additional conditions. 

The class of QNAs can be thought of as a large axiomatically defined class of algebras, modelled on the idea of deforming the enveloping algebra of a finite-dimensional nilpotent Lie algebra. QNAs include, for a symmetric Kac--Moody Lie algebra $\mathfrak g$ with triangular decomposition $\mathfrak g=\mathfrak g^+\oplus \mathfrak h\oplus \mathfrak g^-$, the quantum Schubert cell algebras $U_q(\mathfrak g^+ \cap w\mathfrak g^-)$, where $w$ is an element of the Weyl group and $U_q(\mathfrak g)$ is the corresponding quantum Kac--Moody algebra, with $q$ not a root of unity. In particular, if $\mathfrak g$ is a simple finite-dimensional complex Lie algebra, we obtain $U_q^+(\mathfrak{g}):=U_q(\mathfrak g^+)$ as a QNA. Further examples include quantum matrix algebras, generic quantized coordinate rings of affine, symplectic and euclidean spaces, and generic quantized Weyl algebras. Another interesting class of related examples are quantized coordinate rings of double Bruhat cells of finite-dimensional connected, simply connected complex simple algebraic groups, which are localizations of QNAs. The latter relation to QNAs was exploited in~\cite{GY20} to prove the Berenstein--Zelevinsky conjecture that these quantized coordinate rings admit quantum cluster algebra structures.

Going back to the algebras $U_q^+(\mathfrak g)$, for $\mathfrak g$ simple and finite dimensional of rank $n>1$, Yakimov's Rigidity Theorem \cite[Theorem 5.1]{yakAD} shows that the automorphism group of $U_q^+(\mathfrak g)$ is generated, as a semidirect product, by the torus of rank $n$ acting diagonally on the Chevalley generators, forming an abelian normal subgroup, and the finite group of diagram automorphisms of the Dynkin diagram of $\mathfrak g$. It is thus natural to investigate also the Lie algebra of derivations of $U_q^+(\mathfrak g)$, which can be thought of as infinitesimal transformations, and the corresponding first Hochschild cohomology group $\hh(U_q^+(\mathfrak g))$. To the best of our knowledge, this is generally unknown except for a few specific examples such as $\mathfrak g=\mathfrak{so}_5$ (see~\cite{YL09}) and $\mathfrak g=\mathfrak{sl}_4$ (see \cite{ll}). For the multiparameter case, the Lie algebra of derivations is known for $\mathfrak g=\mathfrak{so}_5$ (see \cite{xt}), $\mathfrak g=\mathfrak{so}_7$ (see \cite{lw}), and for $\mathfrak g=\der(\mathbb O)$, of type $G_2$ (see \cite{zt}). In all of these cases, the strategy used largely involves localization theory and Cauchon's deleting derivations algorithm \cite{cauchon}, along with some \textit{ad hoc} arguments. A similar strategy has also been used to study the derivations of other QNAs, such as the multiparameter quantum Weyl algebra (see~\cite{lR05}), the algebra of quantum matrices $\mathcal O_q (M_n)$ (see \cite{sltl}), and a quantum second Weyl algebra introduced in \cite{lo}. 

In this paper, we determine the derivations and the first Hochschild cohomology group of an arbitrary uniparameter QNA $R$ (as in Definition~\ref{dcgl}) having no central QNA generators and satisfying an additional technical (but necessary) assumption on the center. The latter assumption is satisfied in case all of the normal elements are central and also in case $R=U_q^+(\mathfrak g)$, as above (see Theorem~\ref{thm-Der-QNA}). We find that $\hh(R)$ is a free module over its center $\ZC(R)$, of rank equal to the rank of the maximal torus $\mathcal{H}$ acting rationally by automorphisms on $R$. In fact, we see that $\hh(R)$ can be identified with the $\ZC(R)$-module $\Hom_\Z(X(\mathcal{H}), \ZC(R))$, where $X(\mathcal{H})$ is the character group of $\mathcal{H}$. In particular, for $R=U_q^+(\mathfrak{g})$, where $\mathfrak g$ is simple and finite-dimensional of rank $n>1$, we find that $\hh(R)$ has a free $\ZC(R)$-basis given by the homogeneous derivations $\left\{D_i\right\}_{i=1}^n$ satisfying $D_i (E_j)=\delta_{ij}E_j$, where $E_1, \dots ,E_n$ are the Chevalley generators. We note that, by a recent result of Bell and Buzaglo \cite{bb}, the enveloping algebra of  $\der(R)$ is not noetherian.

We remark that, although the derivations $D_i$ are not locally nilpotent, over the field of complex numbers one can compute, for $\lambda\in\mathbb C$,
\begin{equation*}
e^{\lambda D_i}:=\sum_{k\geq 0} \frac{\lambda^k}{k!}D_i^k,
\end{equation*}
which is the automorphism of $U_q^+(\mathfrak{g})$ defined by $e^{\lambda D_i}(E_j)=
\begin{cases}
e^\lambda E_i & \text{if $i=j$;}\\
E_j & \text{if $i\neq j$.}
\end{cases}
$ Thus, the \textit{exponential map}  
\begin{equation}\label{E:intro:exp}
\mathsf{span}_{\mathbb C}\{ D_1, \ldots, D_n\} \longrightarrow \aut(U_q^+(\mathfrak{g})), \quad D\mapsto e^D
\end{equation}
surjects onto the normal subgroup $(\mathbb C^*)^n$ of $\aut(U_q^+(\mathfrak{g}))$, which is just the maximal torus $\mathcal H$. For a Lie algebra $\mathfrak g$ as above, the center $\ZC(U_q^+(\mathfrak{g}))$ is a nontrivial polynomial algebra, so the Lie algebra $\hh(U_q^+(\mathfrak{g}))$, which we can identify with $\bigoplus_{i=1}^n \ZC(U_q^+(\mathfrak{g})) D_i$, is infinite dimensional and in general nonabelian. By contrast, the Lie algebra $\mathsf{span}_{\mathbb C}\{ D_1, \ldots, D_n\}$ appearing in \eqref{E:intro:exp} is an $n$-dimensional abelian Lie algebra. 

The methods developed in the present paper can be adapted and extended in different directions. First, we can use similar techniques to show that the derivations of certain primitive quotients of QNAs considered in this paper are all inner, thus generalizing results from \cite{lo2} in the case $\mathfrak g=\mathfrak{so}_5$ and from \cite{lo} for $\mathfrak g$ of type $G_2$.  Secondly, we can adapt our techniques to compute the derivations of quantum (upper) cluster algebras. Finally, we can describe the Poisson derivations of Poisson Nilpotent Algebras (also known as Poisson CGLs) and their Poisson primitive quotients under assumptions mirroring the hypotheses of the present paper as well as the Poisson derivations of various Poisson cluster algebras. We will come back to these results in forthcoming publications.

The paper is organized as follows. In Section~\ref{section: derivations partially localized q-affine spaces} we introduce quantum affine spaces and quantum tori, and prove a general result on derivations of polynomial extensions of a finitely generated algebra, leading to an extension of \cite[Corollay 2.3]{op} to partially localized quantum affine spaces. Then, in Section~\ref{S:rec-QNA-localizations}, after recalling the definition of a QNA $R$, we review the construction introduced in \cite{yak} by Goodearl and Yakimov of the set $\{y_1, \ldots, y_N\}$ ($N$ is the Gelfand--Kirillov dimension of $R$) of elements which generate a quantum affine space $\mathcal{A}$, giving rise to a chain of embeddings
\begin{equation*}
\mathcal{A}\subseteq R\subseteq \mathcal{T}\subseteq \fract(R),
\end{equation*} 
where $\mathcal{T}$ is the quantum torus associated to $\mathcal{A}$ and $\fract(R)$ is the skew-field of fractions of $R$. In other words, using the language of cluster algebras, the set $\{y_1, \dots, y_N\}$ is an (initial) quantum cluster for $R$. The elements in $\{y_1, \ldots, y_N\}$ include the homogeneous prime elements of $R$ (which generate a quantum affine space of Gelfand--Kirillov dimension $n$) and give rise to several important localizations of $R$. The intersection of a family of these localizations is shown to equal $R$ (see also Appendix~\ref{App}), a result which will play a crucial role in the proofs of our main results in Section~\ref{section: derivation QNA}. Our main theorem holds under a few additional conditions which are satisfied in particular in case $R= U_q^+(\mathfrak{g})$, for $\mathfrak{g}$ simple of rank $n>1$. Our method consists in first localizing $R$ at an Ore set generated by some of the $y_i$ in such a way that the center of the localization remains equal to $\ZC(R)$. Then we use Corollary~\ref{prop: partially-localized-q-affine-spaces} and our previous result on intersections of localizations of $R$ to conclude that any derivation of $R$ can be decomposed as $\ad_x+\theta$, for some $x\in R$ and a derivation $\theta$ such that $\theta(y_i)\in\ZC(R)y_i$ (if $y_i$ is not central) or $\theta(y_i)\in\ZC(R)$ (if $y_i$ is central), for all $1\leq i\leq N$. Finally, we use the $X(\mathcal{H})$-grading of $R$ to show that the derivations $\theta$ are in one-to-one correspondence with the $\ZC(R)$-module $\Hom_\Z(X(\mathcal{H}), \ZC(R))$, leading to our main results, Theorem~\ref{thm-Der-QNA} on the structure of the space of derivations of $R$, and Corollary~\ref{cor-HH1-QNA}, stating that the first cohomology group $\hh(R)$ is a free $\ZC(R)$-module of rank $n$. We end this section with a series of examples which illustrate the indelible role of the hypotheses in our work. In the final Section~\ref{u+}, we apply our conclusions to the QNAs $U_q^+(\mathfrak{g})$ as above, which are stated in Theorem~\ref{finalder}.

\subsection{Notation and conventions}

Throughout this paper, we work over an arbitrary base field $\k$ of characteristic $0$. In particular, unless otherwise stated, all endomorphism and skew-derivations of $\k$-algebras are assumed to be $\k$-linear. Given integers $i, j\in\Z$, we set $[i,j]:=\{k\in\Z\mid i\leq k\leq j\}$.

As usual, an element of a list appearing with a hat is omitted from this list.

For a $\k$-algebra $A$, we will denote its center by $\ZC(A)$ and its group of $\k$-linear automorphisms by $\aut(A)$. The Lie algebra of $\k$-derivations of $A$ is denoted by $\der(A)$ and its Lie ideal of inner derivations is $\innder(A):=\{\ad_x\mid x\in A\}$, so that $\hh(A)=\der(A)/\innder(A)$ is the first Hochschild cohomology group of $A$. We say that the elements $a, b\in A$ \textit{quasi-commute} if there is some $\xi\in\k^*:=\k\setminus\{0\}$ such that $ab=\xi ba$.

Quantum nilpotent algebras (QNAs, for short) are given in Definition~\ref{dcgl} as iterated Ore extensions. For readers less familiar with Ore extensions (also known as skew polynomial rings) we refer to \cite{GW} and \cite{bg}.

\section{Structure theorem for derivations of partially localized quantum affine spaces}
\label{section: derivations partially localized q-affine spaces}

Derivations of quantum tori were computed in \cite[Corollary 2.3]{op}, where it was proved that every derivation can be expressed uniquely as the sum of an inner derivation and a central derivation, that is, a derivation that acts by multiplication by central elements on the canonical generators of the quantum torus. In the same article, derivations of partially localized quantum affine spaces were proved to be the sum of an inner derivation and a scalar derivation under the assumption that the associated quantum torus is centerless (or, equivalently, simple by \cite[Proposition 1.3]{McConnellPettit1988}), see \cite[Corollary 2.6]{op}. In this section, we compute the derivations of partially localized quantum affine spaces without this simplicity condition but with the assumption that non-localized generators are central. This is the assumption we need in Section \ref{section: derivation QNA}. This assumption is somehow natural since every uniparameter quantum torus is isomorphic, in the generic case, to a commutative Laurent polynomial ring over a simple quantum torus~\cite[Proposition 2.3]{Richard2002}.

Before we state the main results of this section, we fix the notation.

Let $\textbf{q}:=(q_{ij})\in M_n(\k^*)$ be a multiplicatively skew-symmetric matrix, that is, $q_{ii}=1$ and $q_{ij}=q_{ji}^{-1}$ for all $i,j$. 

The quantum affine space associated to $\textbf{q}$, denoted by $\mathcal{A}_\textbf{q}:=\k_{\textbf{q}}[T_1, \ldots, T_n]$, is the $\k$-algebra generated by $T_1$, ..., $T_n$ subject to the relations:  
$$T_jT_i=q_{ij}T_iT_j$$
for all $i,j$. 

Quantum affine spaces are well-understood algebras. They can be presented as iterated Ore extensions over $\k$, and so they are noetherian domains, and the monomials $\underline{T}^{\underline{\alpha}}:=T_{1}^{\alpha_1} \cdots T_{n}^{\alpha_n} $, with $\underline{\alpha}:=(\alpha_1, \dots,\alpha_n) \in (\Z_{\geq 0})^n$, form a basis of $\mathcal{A}_\textbf{q}$ as a $\k$-vector space. 

It is easy to check that each generator $T_i$ is a (regular) normal element of $\mathcal{A}_\textbf{q}$, and the set $E:=\left\{ \lambda  T_{1}^{\alpha_1} \cdots T_{n}^{\alpha_n} \mid \lambda \in \k^*, ~ \alpha_1, \dots,\alpha_n \in \Z_{\geq 0}\right\} $ satisfies the Ore conditions on both sides. The resulting localization 
$$\mathcal{T}_\textbf{q}:=\k_{\textbf{q}}[T_1^{\pm 1}, \ldots, T_n^{\pm 1}]= \mathcal{A}_\textbf{q} E^{-1}$$
is referred to as the quantum torus associated to the multiplicatively skew-symmetric matrix $\textbf{q}$. 

We start with a general result which, for lack of a reference, we include here. It will be used to prove the useful corollary at the end of this section (see also the comment preceding Example~\ref{ex:qweylalg:2}).

\begin{thm}
\label{T:generic}
Let $A$ be a $\k$-algebra and set $R=A[X_1, \ldots, X_m]$, the polynomial algebra over $A$ on $m$ commuting variables. Then
\begin{equation*}
\ZC(R)=\ZC(A)[X_1, \ldots, X_m]\simeq \ZC(A)\otimes_{\k}\k[X_1, \ldots, X_m]. 
\end{equation*}

Additionally, assume that $A$ is finitely generated and that
\begin{align*}
\der(A)=\innder(A)\oplus M, 
\end{align*}
for some $\ZC(A)$-module $M$. Then
\begin{equation*}
\der (R) = \innder (R) \oplus \overline M \oplus \bigoplus_{j=1}^m \ZC(R) \partial_j, 
\end{equation*}
where
\begin{itemize}
\item $\overline M=\ZC(R)M\simeq \ZC(R)\otimes_{\ZC(A)} M\simeq \k[X_1, \ldots, X_m]\otimes_\k M$, so that $D\in M\subseteq\der(A)$ is extended to a derivation of $R$ by setting $D(X_i)=0$, for all $i\in [1,m]$;
\item $\partial_j$ is the derivation of $R$ defined by $\partial_j(A)=0$ and $\partial_j(X_i)=\delta_{ij}$, for all $i, j\in [1,m]$.
\end{itemize}
\end{thm}
\begin{proof}
First, let $z=\sum z_{\underline{\alpha}} \underline{X}^{\underline{\alpha}}$ be a central element of $R$, where the sum runs over all $\underline{\alpha}=(\alpha_1, \dots, \alpha_m) \in (\Z_{\geq 0})^m$ and where all but a finite number of  $z_{\underline{\alpha}} \in A$ are zero. 
Since the $X_j$ are central, one can easily check that $z$ is central if and only if $z a = a z$ for all $a\in A$, that is, if and only if $z_{\underline{\alpha}} a=a z_{\underline{\alpha}}$ for all $a\in A$ and all $z_{\underline{\alpha}}$. 
Thus, $z$ is central if and only if $z_{\underline{\alpha}} \in \ZC(A)$ for all $\underline{\alpha}$, and the first claim follows. 

It is easy to check that the $\partial_j$ define derivations of $R$ and that a derivation $D$ of $A$ can be uniquely extended to a derivation of $R$ by setting $D(X_i)=0$ for all $i\in [1,m]$. Moreover, under such an extension, $zD\in\der(R)$ for all $z\in \ZC(R)$.

Now, let $D\in\der(R)$. Since $D(\ZC(R))\subseteq \ZC(R)$, there are $z_j\in \ZC(R)$ such that $D(X_j)=z_j$ for all $j\in [1,m]$. Thus, replacing $D$ with $D-\sum_{j=1}^m z_j \partial_j\in\der(R)$, we can assume, without loss of generality, that $D(X_i)=0$ for all $i\in[1,m]$. For $a\in A$, we can write \[D(a)=\sum_{\underline{\alpha}} D_{\underline{\alpha}}(a)\underline{X}^{\underline{\alpha}},\] a finite sum with $D_{\underline{\alpha}}(a)\in A$ for all ${\underline{\alpha}}$. It is straightforward to check that the maps $D_{\underline{\alpha}}$ are in fact $\k$-derivations of $A$. Since $A$ is finitely generated as a $\k$-algebra, it follows that there is a finite set $K\subseteq (\Z_{\geq 0})^m$ such that
\begin{align*}
D=\sum_{\underline{\alpha}\in K} \underline{X}^{\underline{\alpha}}D_{\underline{\alpha}},
\end{align*}
where $D_{\underline{\alpha}}$ is extended to a derivation of $R$ as explained above, with $D_{\underline{\alpha}}(X_i)=0$ for all $i\in [1,m]$. By hypothesis, for each $\underline{\alpha}\in K$, there exist $u_{\underline{\alpha}}\in A$ and $E_{\underline{\alpha}}\in M$ such that $D_{\underline{\alpha}}=\ad_{u_{\underline{\alpha}}}+E_{\underline{\alpha}}$. Putting all these together shows that
\begin{align*}
D=\ad_u+\sum_{\underline{\alpha}\in K} \underline{X}^{\underline{\alpha}} E_{\underline{\alpha}}\in\innder(R)+\overline M,
\end{align*}
where $u=\sum_{\underline{\alpha}\in K}u_{\underline{\alpha}}\underline{X}^{\underline{\alpha}} \in R$.

At this stage, we have proved that
\begin{align*}
\der (R) = \innder (R) + \overline M + \sum_{j=1}^m \ZC(R) \partial_j.
\end{align*}

To prove the direct sum decomposition in the statement, assume that 
\begin{align}\label{E:1:T:general}
\ad_u+E+ \sum_{j=1}^m z_j \partial_j=0,
\end{align}
with $u\in R$, $z_j\in \ZC(R)$ and $E=\sum_{\underline{\alpha}} \underline{X}^{\underline{\alpha}} E_{\underline{\alpha}}$, a finite sum with $E_{\underline{\alpha}}\in M$ for every ${\underline{\alpha}}$. Evaluating~\eqref{E:1:T:general} at $X_k$ leads to $z_k=0$, for all $k\in [1,m]$.

Write $u=\sum_{\underline{\alpha}}u_{\underline{\alpha}}\underline{X}^{\underline{\alpha}}$, a finite sum with $u_{\underline{\alpha}}\in A$, for all ${\underline{\alpha}}$. It follows that $\sum_{\underline{\alpha}}\underline{X}^{\underline{\alpha}}\left(\ad_{u_{\underline{\alpha}}}+E_{\underline{\alpha}} \right)=0$. Since $u_{\underline{\alpha}}\in A$ and $E_{\underline{\alpha}}\in M$, we have $\left(\ad_{u_{\underline{\alpha}}}+E_{\underline{\alpha}} \right)(A)\subseteq A$, thus evaluating at an arbitrary $a\in A$ we deduce that $\ad_{u_{\underline{\alpha}}}+E_{\underline{\alpha}}=0$ as a derivation of $A$, for all $\underline{\alpha}$. Now, from $\der(A)=\innder(A)\oplus M $ we deduce that $\ad_{u_{\underline{\alpha}}}=0=E_{\underline{\alpha}}$ as derivations of $A$. But also $\ad_{u_{\underline{\alpha}}}(X_i)=0=E_{\underline{\alpha}}(X_i)$ for all $i\in [1,m]$, so $\ad_{u_{\underline{\alpha}}}=0=E_{\underline{\alpha}}$ as derivations of $R$, for all $\underline{\alpha}$. We conclude that
\begin{align*}
\ad_u= \sum_{\underline{\alpha}}\underline{X}^{\underline{\alpha}}\ad_{u_{\underline{\alpha}}}=0=
\sum_{\underline{\alpha}}\underline{X}^{\underline{\alpha}} E_{\underline{\alpha}}=E,
\end{align*}
as desired. 
\end{proof}

Recall that a quantum torus $\mathcal{T}_\textbf{q}$ is simple if and only if its center is reduced to $\k$, by \cite[Proposition 1.3]{McConnellPettit1988}.

\begin{cor}
\label{prop: partially-localized-q-affine-spaces}
Let $\mathcal{T}_\textbf{q}:=\k_{\textbf{q}}[T_1^{\pm 1}, \ldots, T_n^{\pm 1}]$ be a simple quantum torus. Set $R:= \mathcal{T}_\textbf{q}[X_1, \dots , X_m]$, a commutative polynomial ring over $\mathcal{T}_\textbf{q}$. Then: 
\begin{enumerate}[(a)] 
\item $\ZC(R)= \k[X_1, \dots, X_m]$;
\item $\der (R) = \innder (R) \oplus \bigoplus_{i=1}^n \ZC(R) D_i \oplus \bigoplus_{j=1}^m \ZC(R) \partial_j$, where $D_i$ and $\partial_j$ are the derivations of $R$ defined by: 
\begin{equation*}
D_i(T_k)= \delta_{ik}T_k \mbox{ and } D_i(X_k)=0; 
\end{equation*}
\begin{equation*}
\partial_j(T_k)= 0 \mbox{ and } \partial_j(X_k)=\delta_{jk}. 
\end{equation*}
\end{enumerate}
\end{cor} 
\begin{proof}
The proof follows directly from Theorem~\ref{T:generic} applied to $A=\mathcal{T}_\textbf{q}$, by noting that $\mathcal{T}_\textbf{q}$, being simple, has trivial center, and invoking \cite[Corollary 2.3]{op}, which shows that every derivation of a simple quantum torus is uniquely the sum of an inner derivation and a scalar derivation, that is, a derivation that acts by scalar multiplication on the generators of the quantum torus.
\end{proof}


\section{Prime elements and localizations of QNAs}\label{S:rec-QNA-localizations}

In this section, using the algorithmic construction, due to Goodearl--Yakimov \cite{yak}, of homogeneous elements $y_1, \dots , y_N$ of a QNA $R$, we consider certain localizations of $R$ at Ore sets generated by some of these elements. After proving a technical result that shows that $0$ is the only normal element that is a multiple of a non-normal $y_i$, we prove the main result of the section on intersections of certain localizations of $R$. 

\subsection{Homogeneous prime elements}
\begin{defn}
\label{dcgl}
Suppose that a ring $R$ can be written as an iterated Ore extension of length $N$  as follows:
\begin{equation}
\label{T3}
R=\k[x_1][x_2;\sigma_2,\delta_2]\cdots
[x_N;\sigma_N,\delta_N],
\end{equation} 
where, for $k\in [1,N]$, $\sigma_k$ and $\delta_k$ are, respectively, $\k$-linear automorphisms and $\sigma_k$-derivations of 
\begin{equation}
\label{R_k} 
R_{k-1}:=\k[x_1][x_2;\sigma_2,\delta_2]\cdots
[x_{k-1};\sigma_{k-1},\delta_{k-1}], \mbox{ with } R_0:=\k.
\end{equation}

This iterated Ore extension $R$ is said to be a \textit{quantum nilpotent algebra} (QNA) if there exists a torus $\mathcal{H}=(\k^*)^m$ that acts rationally  by 
$\k$-automorphism on $R$ such that $x_{1},\ldots, x_{N}$ are $\mathcal{H}$-eigenvectors, and the  following are satisfied:
\begin{enumerate}[(i)]
\item for all $k\in [2,N]$ and $k>j$, we have that $\sigma_k(x_j)=\lambda_{kj}x_j$ for some $\lambda_{kj}\in \k^*$.
\item for every $k\in [2,N],$ the $\sigma_k$-derivation $\delta_k$ is locally nilpotent on the subalgebra $R_{k-1}$ of $R$.
\item for every $k\in [1,N],$ there exists $h_k\in \mathcal{H}$ and some $q_k\in \k^*$ which is not a root of unity such that $(h_k\cdot)\mid R_{k-1}=\sigma_k$ and $h_k\cdot x_k=q_kx_k.$
\end{enumerate}

If there exist $q \in \k^*$ not a root of unity and a skew-symmetric integer matrix $A =(a_{ij}) \in \mathcal{M}_N(\Z)$ such that $\lambda_{kj} =q^{a_{kj}}$ for all $j<k$, then $R$ is a {\em uniparameter} QNA. 
\end{defn}
Note that in  the original definition  \cite[Definition 3]{stl} there is the additional condition that there exist $q_k\in \k^*$ not a root of unity such that $\sigma_k \delta_k=q_k\delta_k \sigma_k.$ However, this condition was later proved to follow from the ones listed above  (see  \cite[(3.1)]{yak} for the necessary details).

Observe that, for degree reasons, the group of invertible elements of a QNA is reduced to $\k^*.$

Let $n$ be the {\em rank} of the QNA $R$; that is, $\rk(R):=|\{i \in [1,N] \mid \delta_i=0\}|=n$. The rank of $R$ is also equal to the number of height one prime ideals of $R$ which are invariant under $\mathcal{H}$, see \cite[(4.3)]{yak}. It follows from \cite[Theorem~5.3]{yak} that we can (and will) assume that $m=n$, so that $\mathcal{H}=(\k^*)^n$. In other words, we assume that $\mathcal{H}$ is the largest torus giving $R$ a QNA structure.

Let $X(\mathcal{H})$ denote the set of all rational characters of the torus $\mathcal{H}.$ Then, $X(\mathcal{H})$
is an abelian group called the \textit{character group} of $\mathcal{H}.$  The action of $\mathcal{H}$ on $R$ induces an $X(\mathcal{H})$-grading of $R.$ The $\mathcal{H}$-eigenvectors are exactly the non-zero homogeneous elements under this grading (see \cite[Section 3.2]{yak1}). An element $u\in R$ is \textit{normal} if 
$uR=Ru.$ A non-zero normal element $p\in R$ is said to be a \textit{prime element} if the ideal $pR$ is completely prime.  Finally, a prime element $p\in R$ that is also an  $\mathcal{H}$-eigenvector is simply called a \textit{homogeneous prime element} or a \textit{prime $\mathcal{H}$-eigenvector}.

The algorithmic construction due to Goodearl--Yakimov of the homogeneous prime elements relies on the existence of a \textit{colouring map} $\mu: [1, N] \rightarrow [1,n]$. Attached to such a map, one can define two functions, the \textit{predecessor} function $p=p_\mu:[1,N]\rightarrow [1,N]\sqcup \{-\infty\}$ and the \textit{successor} function $s=s_\mu:[1,N]\rightarrow [1,N]\sqcup \{+\infty\}$ by: 
$$
p(k)=\begin{cases}
\text{max} \ \{j<k\mid \mu(j)=\mu(k)\} & \text{if $\exists j<k$ such that $\mu(j)=\mu(k),$}\\
-\infty & \text{otherwise,}
\end{cases}
$$ and 
$$
s(k)=\begin{cases}
\text{min} \ \{j>k\mid \mu(j)=\mu(k)\} & \text{if $\exists j>k$ such that $\mu(j)=\mu(k),$}\\
+\infty & \text{otherwise.}
\end{cases}
$$

In \cite{yak}, the authors construct a colouring map $\mu: [1, N] \rightarrow [1,n]$ and use it to describe the homogeneous prime elements of a QNA. We recall their result below.

\begin{thm} \cite[Theorem 4.3]{yak}
\label{homo}
Let $R$ be a QNA of rank $n$ as in \eqref{T3}. There exists a surjective function $\mu: [1, N] \rightarrow [1,n]$ such that the following homogeneous elements $y_1, \ldots, y_N$ of $R$ can recursively and uniquely be constructed as follows:
\begin{equation}
\label{5T}
y_k:=\begin{cases}
y_{p(k)}x_k-c_k, & \text{if $p(k)\neq -\infty$},\\
x_k, & \text{if $p(k)= -\infty$},
\end{cases}
\end{equation}
for some $c_k\in R_{k-1}.$ The elements $y_1, \ldots, y_N$  satisfy the  property that, for every $k\in [1,N],$ we have
\begin{equation}
\label{T5}
\{y_j\mid j\in [1,k], \  s(j)>k\}
\end{equation}
is the set of homogeneous prime elements  of $R_k$, up to scalar multiplication.
\end{thm}

We record additional properties of the elements $y_k$ in the following remark.  

\begin{rem}\hfill
\label{T25}
\begin{itemize}
\item[1.] $\delta_k=0$ if and only if $p(k)=-\infty$ (see \cite[Theorem~3.6]{yak1}). 
\item[2.] Assume $p(k) \neq -\infty$. Then it follows from \cite[Proposition 4.7]{yak} that $c_k=\alpha_{k,p(k)}^{-1}(q_k-1)^{-1}\delta_k(y_{p(k)}),$ where $\sigma_k(y_{p(k)})=\alpha_{k,p(k)} y_{p(k)}$ (and $\alpha_{kj}$ is a product of $\lambda_{ki}$, by  \cite[4.15]{yak}). 
\item[3.] Assume $p(k) \neq -\infty$. Then $y_{p(k)}$ is a homogeneous prime element of $R_{k-1}$ as $s(p(k))=k>k-1.$ Hence, $y_{p(k)}R_{k-1}$ is a completely prime ideal of $R_{k-1}.$ 
\item[4.] Assume $p(k) \neq -\infty$. Then  $c_k\not\in y_{p(k)}R_{k-1}$ (see \cite[Theorem 3.6(ii)]{yak}). 

\end{itemize}
\end{rem}

\subsection{Partially localized quantum affine space associated to a QNA}\label{S:rec-QNA-localizations:SSplqas}

From \cite[Theorem 4.6]{yak}, the subalgebra $\mathcal{A}_\textbf{q}$ of $R$ generated by the homogeneous elements $y_1, \ldots, y_N$ is a quantum affine space associated to some multiplicatively skew-symmetric matrix $\textbf{q}:=(q_{ij})\in M_N(\k^*)$. (The entries $q_{ij}$ of $\textbf{q}$ are products of the defining parameters $\lambda_{kl}$, by \cite[4.16]{yak}.)  Thus, 
\begin{equation}
\label{b2}
\mathcal{A}_\q:=\k_{\q}[y_1, \ldots, y_N]
\end{equation}
is a quantum affine space with $y_jy_i=q_{ij}y_iy_j$, for all $i,j\in[1,N]$. We denote by 
\begin{equation}
\label{torus}
\mathcal{T}_\textbf{q}:=\mathcal{A}_\textbf{q}[y_1^{-1}, \ldots, y_N^{-1}]=\k_\textbf{q}[y_1^{\pm 1}, \ldots, y_N^{\pm 1}]
\end{equation}
the quantum torus associated to $\mathcal{A}_\textbf{q}$.

It was proved in  \cite[Theorem 4.6]{yak} that the elements $y_1, \dots, y_N$ form an initial quantum cluster for $R$ in the sense that 
\begin{equation}
\label{T7}
\mathcal{A}_{\textbf{q}}\subseteq R\subseteq \mathcal{T}_\textbf{q}\subseteq \fract(R),
\end{equation} 
where $\fract(R)$ is the skew-field of fractions of $R$. The relationship between $R$ on one hand and $\mathcal{A}_{\textbf{q}}$ (or $\mathcal{T}_{\textbf{q}}$) on the other hand is actually stronger, as we shall see below. 

As we will require localizations of $R$ at multiplicative sets generated by various subsets of $\{y_1, \ldots, y_N\}$, we introduce some notation. Given $I \subseteq [1,N]$, we set $Y_I:=\{y_{k}\mid k\in I\}$ and we denote by $E_I$ the multiplicative system of $R$ generated by $Y_I$. This is an Ore set by \cite[Section 7.1]{yak1}. Moreover, let 
\begin{equation*}
\mathfrak{s}_{<+\infty}:=\{k\in [1, N]\mid s(k)<+\infty\} \quad\text{and}\quad \mathfrak{s}_{+\infty}:=\{k\in [1, N]\mid s(k)=+\infty\}
\end{equation*}
and set $Y_{<+\infty}:=Y_{\mathfrak{s}_{<+\infty}}$, $E_{<+\infty}:=E_{\mathfrak{s}_{<+\infty}}$, $Y_{+\infty}:=Y_{\mathfrak{s}_{+\infty}}$ and $E_{+\infty}:=E_{\mathfrak{s}_{+\infty}}$. We then deduce from \cite[(7.1)]{yak1} that 
\begin{equation}
\label{R-torus}
R[E_{[1,N]}]^{-1}=R[y_1^{-1}, \ldots, y_N^{-1}]=\mathcal{T}_\mathrm{\bf{q}}.
\end{equation}

Since $E_{<+\infty}$ consists of elements which are normal in $\mathcal{A}_{\textbf{q}}$, it constitutes an Ore set in $\mathcal{A}_{\textbf{q}}$, and one can form the partially localized quantum affine space $\mathcal{A}_{\q}E_{<+\infty}^{-1}$. Moreover, we deduce from~\cite[(7.1)]{yak1} that $E_{<+\infty}$ is an Ore set in $R$, and a straightforward induction using~\eqref{5T} shows that $RE_{<+\infty}^{-1}=\mathcal{A}_{\q}E_{<+\infty}^{-1}$. More generally, suppose that $\mathfrak{s}_{<+\infty}\subseteq I\subseteq [1,N]$. Then $E_{<+\infty}\subseteq E_I$ and it is clear that we still have $RE_{I}^{-1}=\mathcal{A}_{\q}E_{I}^{-1}$; so we have the following tower of algebras: 
 
\begin{equation}
\label{T20}
\mathcal{A}_{\q}\subseteq R\subseteq RE_{I}^{-1}=\mathcal{A}_{\q}E_{I}^{-1}\subseteq \mathcal{T}_{\q}\subseteq \fract(R).
\end{equation}

This link between the QNA $R$ and the partially localized quantum affine spaces $\mathcal{A}_{\q}E_{I}^{-1}$, for appropriate choices of $\mathfrak{s}_{<+\infty}\subseteq I\subseteq [1,N]$, will allow us to use the results in Section~\ref{section: derivations partially localized q-affine spaces} on the derivations of partially localized quantum affine spaces to compute derivations of QNAs.

\subsection{Normal elements cannot be multiples of a non-prime $y_i$}

We proceed with a result that proves that $0$ is the only normal element that is a multiple of a non-prime $y_i$ (that is, with $s(i)\neq +\infty$). This result will be used later, namely to describe the action of a derivation of $R$ on the generators $x_i$ when we control its action on the homogeneous elements $y_i$, see Proposition~\ref{centralderivation1}, and in Appendix~\ref{App}. 

\begin{lem}
\label{dd}
For all $ i, j\in [1, N]$ with $i\neq j$, we have that $y_i\not\in y_jR.$ 
\end{lem}
\begin{proof}
Let $w\in R$ and $k\in [1,N].$ Denote the degree in  $x_k$ in the expression of  $w$ in the PBW basis of $R$ by $\deg_{x_k}(w).$  Assume by contradiction that there exist $ i, j\in [1, N]$ with $i\neq j$ and $ y_i\in y_jR$. Thus there exists $u\in R$ such that $y_i=y_ju.$ 
  
Suppose first that $i<j.$ Then,  by construction, we have $\deg_{x_j}(y_i)=0$ and   $\deg_{x_j}(y_ju)\geq 1$, a contradiction.

Next, we suppose that $i>j$.  In this case, it follows from \cite[Lemma 7.5]{yak1} that $u=y_iv$ for some $v\in R$. Then $y_i=y_jy_iv$. Since $y_i$ and $y_j$ quasi-commute, this shows the existence of $w\in R$ such that $1=y_j w$. This is impossible for degree reasons since $\deg_{x_j} (y_j)=1$.
\end{proof}

The set $Y_{+\infty}$ of homogeneous prime elements of $R$
generates a unital subalgebra $\NS(R)$ of $R$, called the \textit{normal subalgebra}, with Gelfand--Kirillov dimension $n$ (see \cite[Theorem 4.6]{yak}). Thus, 
\begin{equation}
\label{T9}
\NS(R):=\k_{\q'}[ y_j \mid j\in \mathfrak{s}_{+ \infty} ]\subseteq \mathcal{A}_\textbf{q},
\end{equation}
where $\textbf{q}'$ is a multiplicatively skew-symmetric sub-matrix of $\textbf{q}$.

We are now ready to establish the following technical result.

\begin{pro}
\label{b5}
Let $R$ be a QNA and  $y_i\in R$ be a homogeneous element with $s(i)<+\infty$. Then $\NS(R) \cap y_iR = \{0\}=\NS(R) \cap Ry_i$.
\end{pro}
\begin{proof}
Suppose that $0\neq u\in \NS(R) \cap y_iR$. Then, there exists 
$v\in R$ such that $u=y_i v.$ Write $u=u_1+\cdots+u_d,$ where the $u_j$ are nonzero $\ch$-eigenvectors with different $\ch$-eigenvalues. It follows from \cite[Proposition 6.20]{myak} 
that each $u_j$ is normal. Similarly, one can also decompose $v\in R$ as 
$v_1+\cdots +v_e$, where the $v_j$ are nonzero $\ch$-eigenvectors with different $\ch$-eigenvalues. Returning to $u=y_i v$, we have that 
$u_1+\cdots+u_d=y_iv_1+\cdots+y_iv_e.$ Each $y_i v_j$  is an $\ch$-eigenvector and they all have different $\ch$-eigenvalues. The uniqueness of the decomposition implies that $d=e$, and there exists a permutation $\tau\in S_d$ such that $u_j=y_iv_{\tau(j)}$ for all $j$, with $u_j$ normal and $v_{\tau(j)}\in R$, both $\ch$-eigenvectors.  

Therefore, we can assume that $u=y_iv$ with $u$ and $v$ both $\ch$-eigenvectors and $u$ normal. From \cite[Proposition 3.2]{stl}, we have that the QNA $R$ is an $\ch$-UFD (unique factorization domain), and so it follows from 
\cite[Proposition 2.2]{yak} that $u$ is either a unit or can be decomposed as 
$u=p_1p_2\ldots p_l$ where $l\geq 1$ and each $p_i$ is a homogeneous prime element (a prime $\ch$-eigenvector). Since the invertible elements of $R$ are reduced to non-zero scalars and $y_iR  \cap  \k^*=\emptyset,$ we conclude that $u$ is not a unit. Hence, $u=p_1p_2\ldots p_l$ where each $p_i$ is a homogeneous prime element. It follows that 
$y_iv=p_1p_2\ldots p_l\in p_1R=Rp_1$. Since, by the definition of a prime element, the ideal $Rp_1$ is completely prime, $y_iv\in Rp_1$ implies that either $y_i\in Rp_1$ or $v\in Rp_1.$ Since $p_1$ is a homogeneous prime element, it follows from Theorem \ref{homo} that there exist $\gamma\in \k^*$ and $j\in [1,N]$ with $s(j)=+\infty$ such that $p_1=\gamma y_j$. Hence, $y_i\in Rp_1$ implies that $y_i\in Ry_j=y_jR$. Given that $s(i)<+\infty$ whereas $s(j)=+\infty,$ we have that $i \neq j$ and $y_i\in Ry_j$, contradicting Lemma \ref{dd}. Therefore,
$v\in p_1R$. This implies that $v=p_1v'$ for some $v'\in R.$ 
So, $p_1\ldots p_l=u=y_iv=y_ip_1v'=p_1\lambda_i y_iv'$ for some $\lambda_i\in \k^*$, as $p_1$ is a homogeneous prime element. Consequently, 
$p_2\ldots p_l=\lambda_i y_i v'.$ Repeating the argument above will eventually lead to $1=y_i w$, with $w\in R,$ a contradiction.

The proof that $\NS(R) \cap Ry_i=\{0\}$ is symmetric.
\end{proof}

\subsection{Intersections of localizations}

Below we have one of the main results in this section.

\begin{thm}
\label{b4}
Let $I,J \subseteq [1,N]$. Then $RE_I^{-1} \cap RE_J^{-1}=RE_{I\cap J}^{-1}$.
\end{thm}
\begin{proof}
In case $R$ is a \textit{symmetric} QNA (see \cite[Definition 3.12]{yak1})
this is a consequence of the fact that each nonzero element of $RE^{-1}$ has a unique minimal
denominator. For the general case, the proof is more technical and is included in Appendix~\ref{App}.
\end{proof}

\section{Centers--the zeroth Hochschild cohomology group}
\label{section: centers}

The center of an algebra is its Hochschild cohomology group of degree zero and it constitutes an important invariant subalgebra which acts on its Lie algebra of derivations and on the first Hochschild cohomology group.
In this section, we will be concerned with the centers of the QNA $R$, the quantum affine space $\mathcal{A}_{\mathrm{\bf{q}}}$, the quantum torus $\mathcal{T}_{\mathrm{\bf{q}}}$, and certain localizations of these.

The first observation is that, since
\begin{equation*}
\ZC(R)\subseteq  \NS(R)\subseteq \mathcal{A}_{\mathrm{\bf{q}}}\subseteq R\subseteq \mathcal{T}_{\mathrm{\bf{q}}},
\end{equation*}
where $\NS(R)$ is the normal subalgebra of $R$ introduced in~\eqref{T9}, it follows that 
\begin{equation}\label{E:chain:center:R:Aq:Tq}
\ZC(R)=\ZC(\mathcal{A}_{\q})=\ZC(\mathcal{T}_{\q})\cap \mathcal{A}_{\q}.
\end{equation}
Note also that, from \cite[Proposition 2.11]{yak2}, we have
$\ZC(\mathcal{T}_\q)\subseteq \NS(R)E_{+\infty}^{-1}$, the quantum torus associated with $\NS(R)$.

Note that, since $R$ is a uniparameter QNA, all the parameters $\lambda_{ij}$ and so all the entries of $\q$ are powers of the parameter $q$.  In other words, $\mathcal{T}_{\q}$ is a uniparameter quantum torus. As a consequence, we have that $\ZC(\mathcal{T}_{\q})=\k[z_1^{\pm 1}, \ldots, z_\ell^{\pm 1}]$, for some $0\leq\ell\leq n$, and the $z_i$ can be chosen to be monomials in the $y_j^{\pm 1}$, with $j\in \mathfrak{s}_{+\infty}$. In case $\ZC(\mathcal{T}_{\mathrm{\bf{q}}})=\k$, which is a possibility, the convention is that $\ell=0$.

We are looking for situations in which we are able to conclude, among other properties, that $\ZC(\mathcal{A}_{\mathrm{\bf{q}}})=\k[z_1, \ldots, z_\ell]$.

\begin{exa}
Let $R=\mathcal A_\q$, the quantum affine space associated with the matrix $\q=
\begin{psmallmatrix}
1& q^2 & q^3\\
q^{-2} &1 & q^5\\
q^{-3} &q^{-5} &1 \\ 
\end{psmallmatrix}
$, where $q\in\k^*$ is not a root of unity. Then $R$ is a uniparameter QNA of rank $3$ with $\ZC(\mathcal{T}_{\mathrm{\bf{q}}})=\k[z^{\pm 1}]$, with $z=x_1^5 x_2^{-3} x_3^{2}$, so $\ZC(\mathcal{A}_{\mathrm{\bf{q}}})=\k$.
\end{exa}

In contrast with the example above, with many other QNAs, including $U_q^+(\mathfrak{g})$ with $\mathfrak{g}$ a finite-dimensional complex simple Lie algebra, it is possible to choose the generators $z_i$ of the Laurent polynomial ring $\ZC(\mathcal{T}_{\q})$ so that $\ZC(\mathcal{A}_{\q})=\k[z_1, \ldots, z_\ell]$. 

For $1\leq i\leq\ell$, set 
\begin{equation*}
\supp z_i=\{j\in \mathfrak{s}_{+\infty}\mid \deg_{y_j}z_i\neq 0\},
\end{equation*}
where $\deg_{y_j}$ is computed in the quantum torus of the normal subalgebra $\NS(R)$. 

We want to be able to identify each central generator $z\in\{z_1, \ldots, z_\ell\}$ by a distinguished element $y_c$ with $c\in\supp z$, which we will call a \textit{pivot}. To be precise, we impose the following hypothesis.

\begin{hypo}\label{hyp}
We assume that there is a choice for the monomial generators $z_1, \ldots, z_\ell$ of the Laurent polynomial ring $\ZC(\mathcal{T}_{\mathrm{\bf{q}}})$ and a subset $C=\{c_1, \ldots, c_\ell\}\subseteq \mathfrak{s}_{+\infty}$, with $|C|=\ell$, such that, for all $1\leq i\leq\ell$:
\begin{enumerate}[\ (H1)]
\item $z_i\in\mathcal{A}_{\mathrm{\bf{q}}}$;\label{hyp:1}
\item $\deg_{y_{c_j}}z_i = \delta_{ij}$; \label{hyp:2}
\item if $|\supp z_i|\geq 2$ then $\supp z_i\setminus{\{c_i\}}\not\subseteq\bigcup_{j\neq i}\supp z_j$.\label{hyp:3}
\end{enumerate}
We call the elements in $Y_C=\{y_{c_1}, \ldots, y_{c_\ell}\}$ \textit{pivots}.
\end{hypo}

\begin{rem}\label{R:rem:on:hyp}\hfill
\begin{enumerate}
\item (H\ref{hyp:2}) above implies that $\supp z_i \cap C=\{c_i\}$.
\item Assuming (H\ref{hyp:2}), it is easy to see that $y_{c_i}\in\ZC(R)\iff\supp z_i=\{c_i\}$. Thus, (H\ref{hyp:3}) could be replaced with the equivalent formulation: 
\begin{enumerate}
\item[\textit{(H'3)}] \textit{if $y_{c_i}$ is not central, then there is $k\in\supp z_i$ such that $k\neq c_i$ and $k\notin \supp z_j$, for any $j\neq i$}.
\end{enumerate}
\item In case all normal elements of $R$ are central, i.e.\ $\NS(R)=\ZC(R)$, then $\ell=n=\rk(R)$ and we can take $\{z_1, \ldots, z_n\}=Y_{+\infty}$ and $C=\mathfrak{s}_{+\infty}$. We see that Hypothesis~\ref{hyp} holds in this case. This covers the QNAs of the form $U_q^+(\mathfrak{g})$, with $\mathfrak{g}$ of type 
$A_1$, $B_n~(n\geq 2)$, $C_n~(n\geq 3)$, $D_n~(n\geq 4 ~\text{even})$, $ G_2$, $F_4$, $E_7$ and $E_8$.
\item More generally, if the $\supp z_i$, with $1\leq i\leq \ell$, are pairwise disjoint and, up to a nonzero scalar factor, $z_i=\prod_{k\in\supp z_i}y_k$, then we can choose any $c_i\in\supp z_i$. We see that Hypothesis~\ref{hyp} holds in this case, which covers all QNAs of the form $U_q^+(\mathfrak{g})$, with $\mathfrak{g}$ simple of any finite type.
\end{enumerate}
\end{rem}

Assume that Hypothesis~\ref{hyp} holds. Let $E:=E_{[1,N]\setminus C}$, the Ore set in $R$ generated by all the $y_i$ that are not pivots. Set
\begin{equation*}
\mathcal{T}_{\widehat\q}:=\k_{\widehat\q}[y_i^{\pm 1}\mid i\in [1,N]\setminus C],
\end{equation*}
the quantum torus of rank $N-\ell$ generated by the non-pivots, where $\widehat\q$ is an appropriate submatrix of $\q$. Finally, set $\widehat{R}=RE^{-1}$.

\begin{pro}\label{P:ZCs}
Assume that Hypothesis~\ref{hyp} holds. Then we have the following:
\begin{enumerate}[(a)]
\item $\widehat{R}=\mathcal{A}_{\q} E^{-1}=\mathcal{T}_{\widehat\q}[z_1, \ldots, z_\ell]$;\label{P:ZCs:1}
\item $\mathcal{T_\q}=\mathcal{T}_{\widehat\q}[z_1^{\pm1}, \ldots, z_\ell^{\pm1}]$;\label{P:ZCs:2}
\item $\ZC(\mathcal{T}_{\widehat\q})=\k$;\label{P:ZCs:3}
\item $\ZC(B)=\k[z_1, \ldots, z_\ell]$, for any subalgebra $B$ such that $\mathcal{A}_\q\subseteq B\subseteq\widehat R$.\label{P:ZCs:4} 
\end{enumerate}
In particular, $\ZC(R)=\ZC(\widehat R)=\k[z_1, \ldots, z_\ell]$ and $N-\ell$ is even.
\end{pro}
\begin{proof}
For~(\ref{P:ZCs:1}) above, recall that we have observed at the end of Subsection~\ref{S:rec-QNA-localizations:SSplqas} that $RE_{<\infty}^{-1}=\mathcal{A}_{\q}E_{<\infty}^{-1}$. Since $E_{<\infty}\subseteq E$, it follows that $\widehat{R}=RE^{-1}=\mathcal{A}_\q E^{-1}=\mathcal{T}_{\widehat\q}[y_c\mid c\in C]$.

It's clear that the elements $z_1, \ldots, z_\ell$ are algebraically independent over $\mathcal{T}_{\widehat\q}$, because the set of variables $Y_C$ is algebraically independent over $\mathcal{T}_{\widehat\q}$ and $\supp z_i \cap C=\{c_i\}$. So $\mathcal{T}_{\widehat\q}[z_1, \ldots, z_\ell]$ is a (commutative) polynomial extension of $\mathcal{T}_{\widehat\q}$ and $\mathcal{T}_{\widehat\q}[z_1, \ldots, z_\ell]\subseteq \mathcal{T}_{\widehat\q}[y_c\mid c\in C]$.

Conversely, given $1\leq i\leq \ell$, Hypothesis~\ref{hyp} implies that, up to a nonzero scalar factor, $z_i=y_{c_i}v_i$, where $v_i$ is a monomial in the $y_k$ with $k\notin C$. So $v_i^{\pm1}\in \mathcal{T}_{\widehat\q}$ and $y_{c_i}=z_iv_i^{-1}\in \mathcal{T}_{\widehat\q}[z_1, \ldots, z_\ell]$, establishing the other inclusion.

Now~(\ref{P:ZCs:2}) follows from~(\ref{P:ZCs:1}), as 
\begin{equation*}
\mathcal{T_\q}= \mathcal{A}_{\q} E^{-1}E_C^{-1}=\mathcal{T}_{\widehat\q}[z_1, \ldots, z_\ell]E_C^{-1}\subseteq \mathcal{T}_{\widehat\q}[z_1^{\pm1}, \ldots, z_\ell^{\pm1}],
\end{equation*}
where the last inclusion follows from the relation $y^{-1}_{c_i}=v_iz_i^{-1}$, for some $v_i \in \mathcal{T}_{\widehat\q}$. The inclusion $\mathcal{T}_{\widehat\q}[z_1^{\pm1}, \ldots, z_\ell^{\pm1}]\subseteq \mathcal{T_\q}$ is evident.

To show~(\ref{P:ZCs:3}) note that, by~(\ref{P:ZCs:2}), 
\begin{equation*}
\k[z_1^{\pm1}, \ldots, z_\ell^{\pm1}]=\ZC(\mathcal{T_\q})=\ZC(\mathcal{T}_{\widehat\q})[z_1^{\pm1}, \ldots, z_\ell^{\pm1}].
\end{equation*}
So indeed it must be that $\ZC(\mathcal{T}_{\widehat\q})=\k$. It is well known that the center of an odd rank uniparameter quantum torus  is non-trivial, see for instance \cite[Proposition 2.3]{Richard2002}. Whence, the triviality of the center of $\mathcal{T}_{\widehat\q}$ forces the rank of $\mathcal{T}_{\widehat\q}$ to be even. So $|[1,N]\setminus C|=N-\ell$ is even.

It remains to prove~(\ref{P:ZCs:4}). By~(\ref{P:ZCs:1}) and~(\ref{P:ZCs:3}), $\ZC(\widehat R)=\k[z_1, \ldots, z_\ell]$. As $\widehat{R}=\mathcal{A}_{\q} E^{-1}$, we have 
\begin{equation*}
\ZC(\mathcal{A}_{\q}) = \ZC(\widehat R)\cap \mathcal{A}_{\q}=\k[z_1, \ldots, z_\ell]\cap \mathcal{A}_{\q}=\k[z_1, \ldots, z_\ell].
\end{equation*}
If $\mathcal{A}_\q\subseteq B\subseteq\widehat R$ is a subalgebra, then we deduce from $\widehat{R}=\mathcal{A}_{\q} E^{-1}$ that $\ZC(\mathcal{A}_{\q})\subseteq\ZC(B)\subseteq\ZC(\widehat R)$, yielding $\ZC(B)=\k[z_1, \ldots, z_\ell]$.
\end{proof}


\section{The first Hochschild cohomology group of a QNA}
\label{section: derivation QNA}

This is the main section of the paper and it focuses on investigating the first Hochschild cohomology group of a QNA $R$ satisfying the following two conditions: 
\begin{enumerate}[(i)]
\item $R$ is a uniparameter QNA with parameter $q$ as in Definition \ref{dcgl};
\item Hypothesis~\ref{hyp} holds.
\end{enumerate}
Throughout this section, unless otherwise stated, we assume these two hypotheses are satisfied. We will provide interesting examples of such QNAs in the final section of this paper. 

We will show that each derivation of $R$ decomposes (uniquely) as a sum of an inner derivation and a homogeneous derivation (see Subsection~\ref{SS:homogeneous}). We begin by tackling the inner part of a derivation of $R$.

\subsection{The inner component of a derivation of $R$}

To study the space $\der(R)$ of $\k$-derivations of $R$, notice that we can uniquely extend any derivation $D$ of $R$ to a derivation of $\widehat{R}=RE^{-1}$, via localization. This makes it clear that, using the same notation $D$ for this extension, we have $D(\widehat{R})\subseteq \widehat{R}$. In fact, we can identify $\der(R)$ with $\{D\in \der(\widehat R)\mid D(R)\subseteq R\}$.

From Proposition~\ref{P:ZCs} we have that $\mathcal{T}_{\widehat\q}=\k_{\widehat\q}[y_i^{\pm 1}\mid i\in [1,N]\setminus C]$ is a simple quantum torus and $\widehat{R}=\mathcal{T}_{\widehat\q}[z_1, \ldots, z_\ell]$. Thus, Corollary~\ref{prop: partially-localized-q-affine-spaces} shows that, as a derivation of $\widehat{R}$, we can decompose $D$ (uniquely) as
\begin{equation}\label{E:D-as-ad-plus-central}
D=\ad_x+\theta, 
\end{equation}
for some $x\in \widehat{R}$ so that, for all $i\in [1,N]\setminus C$, $\theta(y_i)= \omega_i y_i$, for some $\omega_i \in \ZC(\widehat{R})=\ZC(R)=\k[z_1, \ldots, z_\ell]$. At the end of this subsection, we will also be able to describe the action of $\theta$ on the pivot variables $y_c$, with $c\in C$.

Our first goal, however, is to show that $x\in R$ and, for that purpose, we need to introduce intermediate subalgebras between $R$ and $\widehat R$.

Recall that $E=E_{[1,N]\setminus C}$. For each $k\in [1,N]\setminus C$, let 
$F_k:=E_{[1,N]\setminus (C \cup\{k\})}$ be the Ore set in $R$ generated by $\{y_i\mid i\notin C \cup\{k\}\}$ (see \cite[(7.1)]{yak1}), and 
\begin{equation*}
B_k:=RF_k^{-1} 
\end{equation*}
be the corresponding localization.

Since $E$ satisfies the Ore condition over $R$ and $E$ is generated by $F_k$ and $y_k$, and $y_k$ quasi-commutes with the generators of $F_k$, it follows that $y_k$ generates a multiplicative system that satisfies the Ore condition in $B_k$. Moreover, we have the following chain of embeddings: 
\begin{equation}
\label{T23}
 R\subseteq B_k\subseteq \widehat{R}=B_k[y_k^{-1}]=RE^{-1}\subseteq \mathcal{T}_{\textbf{q}}.
\end{equation}
We know already that $\ZC(R)=\ZC(B_k)=\ZC(\widehat{R})$, by Proposition~\ref{P:ZCs}, so we have complete control over the centers of all algebras appearing in~\eqref{T23}.

For each $k\in [1,N]\setminus C$, let $\mathcal{Q}_{k}=\k_{{\widehat\q}_k}[y_i^{\pm 1}\mid  i\notin C \cup\{k\}]$. So $\mathcal{Q}_{k}$ is a quantum torus, where ${\widehat\q}_k$ is the multiplicatively skew-symmetric matrix obtained from $\q$ by deleting its rows and columns indexed by $C \cup\{k\}$. Moreover, $\mathcal{Q}_{k}\subseteq B_k$ and 
\begin{equation}\label{E:Tqh:over:Qk}
\mathcal{T}_{\widehat{\q}}=\bigoplus_{j\in\mathbb{Z}}\mathcal{Q}_{k} y_k^j. 
\end{equation}

The rank of $\mathcal{Q}_{k}$ is $N-\ell-1$, which is odd, by Proposition~\ref{P:ZCs}. Thus, as the center of an odd rank uniparameter quantum torus is non-trivial (see \cite[Proposition 2.3]{Richard2002}) and central elements in a quantum torus are sums of central (Laurent) monomials in the generators of the quantum torus, we have the following result.
 
\begin{lem}
\label{a17}
For each $k\in [1,N]\setminus C$, there exists a non-trivial monomial $\prod_{i\in [1,N]\setminus (C \cup\{k\})} y_i^{m_i}$ (with at least one integer $m_i \neq 0$) in the center of the  quantum torus 
$\mathcal{Q}_{k}=\k_{{\widehat\q}_k}[y_i^{\pm 1}\mid  i\notin C \cup\{k\}]$.
\end{lem}

Since $B_k$ is a localization of $R$ contained in $\widehat{R}$, we can further think of $D$ as a derivation of $\widehat{R}$ such that $D(R)\subseteq R$ and $D(B_k)\subseteq B_k$.

\begin{lem}\label{L:inner:x-in-R}
Let $x\in \widehat{R}$ be as in~\eqref{E:D-as-ad-plus-central}. Then $x\in R$.
\end{lem}

\begin{proof}
For each $k\in [1,N]\setminus C$,
let $C_k=\mathcal{Q}_{k}[z_1, \ldots, z_\ell]$. Then $C_k$ is a subalgebra of $B_k$ and, by \eqref{E:Tqh:over:Qk},
\begin{equation}\label{E:Rhat:over:Ck}
\widehat{R}=\bigoplus_{j\in\mathbb{Z}}C_k y_k^j. 
\end{equation}
As a result, $x\in \widehat{R}$ can be written uniquely as
$$x=\sum_{j\in \Z}a_{(k,j)}y_{k}^j,$$ 
where $a_{(k,j)}\in C_{k}.$
Decompose $x=x_+ + x_-,$ where
$$x_-=\sum_{j<0}a_{(k,j)}y_{k}^j \quad \text{and} \quad 
x_+=\sum_{j\geq 0}a_{(k,j)}y_{k}^j.$$
Clearly, $x_+\in B_k$.   We now proceed to show that $x_-\in B_k$, for each $k$, by using a strategy already used in the proof of \cite[Proposition 2.3]{ak} (see also the proof of \cite[Lemma 5.9]{lo}). Since $C_k$ is generated by the quantum torus $\mathcal{Q}_{k}$  and the central variables $z_i$, with $1\leq i\leq\ell$, we deduce from Lemma~\ref{a17} that there exists a non-trivial monomial in the generators of $\mathcal{Q}_k$, denoted by $u_k$, that is central in $C_{k}$. Note that $u_k$ does not belong to $\ZC(\widehat{R})$ because $\mathcal{Q}_{k}\cap\ZC(\widehat{R})\subseteq \mathcal{T}_{\widehat{\q}}\cap \k[z_1, \ldots, z_\ell]=\k$, by~(H\ref{hyp:2}). Since the monomial $u_k$ is central in $C_k$ and not in $\widehat{R}$, then~\eqref{E:Rhat:over:Ck} forces $u_{k}y_{k}\neq y_{k}u_{k}$ and so $y_{k}u_{k}=\xi u_{k}y_{k}$, for some $\xi:=\xi_{k}\in\k\setminus\{0,1\}$. In fact, $\xi$ is not a root of unity, as $\xi^\ell=1$ for some $\ell\in\Z$ implies that $y_{k}u_{k}^\ell=u_{k}^\ell y_{k}$, so $u_{k}^\ell\in\mathcal{Q}_k\cap \ZC(\widehat{R})=\k$, forcing $\ell=0$.

Since  $\theta(y_{j})=\omega_{j}y_{j}$
for each $j\in [1,N]\setminus C$, with $\omega_{j}\in \ZC(R)=\ZC(B_k)$,  we have that 
$\theta(u_{k})=\eta_{k} u_{k}$, for some $\eta_{k}\in \ZC(B_k)$. 
Note that $u_{k}^{\pm 1}\in \mathcal{Q}_k\subseteq  B_{k}$.  Since $D$ restricts to a derivation of $B_k$, we have that
\[D(u_{k}^i)=\ad_{x}(u_k^i)+\theta(u_{k}^i)
=\ad_{x_-}(u_{k}^i)+\ad_{x_+}(u_{k}^i)+
i\eta_{k} u_{k}^i\in B_{k},\] 
for all $i\in \mathbb Z$.
Observe that $\ad_{x_+}(u_{k}^i)+
i\eta_{k} u_{k}^i\in B_{k}$. Hence, $\ad_{x_-}(u_{k}^i)\in B_{k}$.
It follows that
\[\ad_{x_-}(u_{k}^i)=\sum_{j=-1}^{-m}(1-\xi^{-ij})a_{(k,j)}y_{k}^ju_{k}^i\in B_{k},\] 
for some $m\in \Z_{>0}$.
Hence, 
\[\chi_i:=\ad_{x_-}(u_{k}^i)u_{k}^{-i}=\sum_{j=-1}^{-m}(1-\xi^{-ij})a_{(k,j)}y_{k}^j\in B_{k},\]
for all $i\in\Z$.

We have the following matrix equation:
\[
\begin{bmatrix}
(1-\xi)&(1-\xi^{2})&\cdots &(1-\xi^{m})\\
(1-\xi^{2})&(1-\xi^{4})&\cdots &(1-\xi^{2m})\\
(1-\xi^{3})&(1-\xi^{6})&\cdots &(1-\xi^{3m})\\
\vdots&\vdots&\ddots&\vdots\\
(1-\xi^{m})&(1-\xi^{2m})&\cdots &(1-\xi^{m^2})\\
\end{bmatrix}
\begin{bmatrix}
a_{(k,-1)}y_{k}^{-1}\\
a_{(k,-2)}y_{k}^{-2}\\
\vdots\\
a_{(k,-m+1)}y_{k}^{-m+1}\\
a_{(k,-m)}y_{k}^{-m}\\
\end{bmatrix}=
\begin{bmatrix}
\chi_1\\
\chi_2\\
\chi_3\\
\vdots\\
\chi_m\\
\end{bmatrix}.
\]
We already know that each $\chi_i$ is an elment of $B_k$. 
We now show that, for each $k\in [1,N]\setminus C$, the elements $a_{(k,j)}y_{k}^j$ also belong to $B_k$, for all $j<0$.
It is sufficient to do this by showing that the coefficient matrix 
\[
U:=\begin{bmatrix}
(1-\xi)&(1-\xi^{2})&\cdots &(1-\xi^{m})\\
(1-\xi^{2})&(1-\xi^{4})&\cdots &(1-\xi^{2m})\\
(1-\xi^{3})&(1-\xi^{6})&\cdots &(1-\xi^{3m})\\
\vdots&\vdots&\ddots&\vdots\\
(1-\xi^{m})&(1-\xi^{2m})&\cdots &(1-\xi^{m^2})\\
\end{bmatrix}
\]
is invertible. 
Apply row operations: $ -r_{m-1}+r_m\rightarrow r_m,\ldots, -r_2+r_3\rightarrow r_3, -r_1+r_2\rightarrow r_2$ to $U$ to obtain: 
$$
U'=\begin{bmatrix}
l_1&l_2&l_3&\cdots &l_m\\
\xi l_1&\xi^{2}l_2&\xi^{3}l_3&\cdots &\xi^{m}l_m\\
\xi^{2}l_1&\xi^{4}l_2&\xi^{6}l_3&\cdots &\xi^{2 m}l_m\\
\vdots&\vdots&\vdots&\ddots&\vdots\\
\xi^{m-1}l_1&\xi^{2(m-1)}l_2&\xi^{3(m-1)}l_3&\cdots &\xi^{m(m-1)}l_m\\
\end{bmatrix},
$$
where $l_i:=1-\xi^{i}$, for $i\in [1,m]$. Since $\xi$ is not a root of unity and $\xi\neq 0$, it follows that $U'$ is similar to a  Vandermonde matrix whose parameters are pairwise distinct. Hence, $U'$ is invertible. Consequently, $U$ is also invertible. 
Therefore, each $a_{(k,j)}y_{k}^j$ is a linear combination of the $\chi_i\in B_k$. Hence, for each $k\in [1,N]\setminus C$, we have that $a_{(k,j)}y_{k}^j\in B_k$, for all $j<0$. 

We can therefore conclude that 
 $x_{-}=\sum_{j=-1}^{-m}a_{(k,j)}y_{k}^j\in B_k,$ and so $x=x_++x_-\in B_k$.  
Consequently,  
\[x\in \bigcap_{k\in [1,N]\setminus C}B_k= \bigcap_{k\in [1,N]\setminus C}R{E^{-1}_{[1,N]\setminus (C \cup\{k\})}}= R,\] 
the last equality following from Theorem \ref{b4}. 
\end{proof}

\begin{cor}
\label{cor-inner}
$\ad_x$ and $\theta=D-\ad_x$ are derivations of $R$.
\end{cor}

Now we can describe the action of $\theta$ on all $y_i$, with $i\in [1, N]$. Note that, if $y_i\in\ZC(R)$, then $\theta(y_i)\in\ZC(R)$, as $\theta\in\der (R)$ and derivations take central elements to central elements.

\begin{cor}\label{C:theta:action:c}
Let $i\in [1, N]$. If $y_i\notin\ZC(R)$ then $\theta(y_i)\in\ZC(R) y_i$.
\end{cor}
\begin{proof}
In case $i\notin C$ the desired conclusion has already been established in \eqref{E:D-as-ad-plus-central}. So let $c\in C$ and assume that $y_c$ is not central. By (H\ref{hyp:1}) and (H\ref{hyp:2}), there is a unique $1\leq i\leq \ell$ such that, up to a nonzero scalar factor, $z_i=y_c v$, where $v$ is a monomial in the $y_k$, with $k\in (\supp z_i)\setminus C\subseteq \mathfrak{s}_{+\infty}\setminus C$.

\underline{Claim:} $\theta(y_c)\in\NS(R)$. 

Since $z_i\in\ZC(R)$, we have
\begin{equation*}
\ZC(R)\ni \theta(z_i)= \theta(y_c) v + y_c \theta(v)=\theta(y_c) v +  y_c\eta v= \theta(y_c) v + \eta z_i,
\end{equation*}
for some $\eta\in\ZC(R)$, because we already know that $\theta(y_k)\in\ZC(R) y_k$, for all $k\notin C$. It follows that $\theta(y_c) v\in\ZC(R)\subseteq \NS(R)$. Thus, as $\theta(R)\subseteq R\subseteq RE_{<+\infty}^{-1}=\mathcal{A}_{\q}E_{<+\infty}^{-1}$, we get
\begin{align*}
\theta(y_c)\in & \NS(R)E_{+\infty}^{-1}\cap R \subseteq \NS(R)E_{+\infty}^{-1}\cap\mathcal{A}_{\q}E_{<+\infty}^{-1} \\
&\quad =\k_{\q'}[ y_k^{\pm 1} \mid  k \in \mathfrak{s}_{+\infty}]\cap \k_\q [y_k\mid k \in \mathfrak{s}_{+\infty}][y_j^{\pm 1}\mid j \in \mathfrak{s}_{<+\infty}]\\
&\quad =\NS(R).
\end{align*}

The claim is proved; now write $\vartheta=\theta(y_c) v\in\ZC(R)$. We can decompose $\vartheta=\vartheta_0+\vartheta_1 z_i$, with $\vartheta_1\in\ZC(R)$ and $\vartheta_0\in\k[z_1, \ldots, \widehat{z_i}, \ldots, z_\ell]$. So
\begin{equation*}
\vartheta_0= \theta(y_c) v-\vartheta_1 z_i= (\theta(y_c) -\vartheta_1 y_c) v\in\NS(R)v\cap\k[z_1, \ldots, \widehat{z_i}, \ldots, z_\ell].
\end{equation*}
Now, by (H\ref{hyp:3}) (see also the ensuing Remark~\ref{R:rem:on:hyp}), there is $k\in\supp z_i$ such that $k\neq c$ and $k\notin \supp z_j$, for any $j\neq i$. So, working in $\NS(R)$, $y_k$ divides $v$, and thus any element in $\NS(R)v$; in contrast, we have that $\k[z_1, \ldots, \widehat{z_i}, \ldots, z_\ell]\subseteq\k_{\q'_k}[y_j\mid s(j)=+\infty,\ j\neq k]$, hence the only element in $\k[z_1, \ldots, \widehat{z_i}, \ldots, z_\ell]$ divisible by $y_k$ is $0$. It follows that $\NS(R)v\cap\k[z_1, \ldots, \widehat{z_i}, \ldots, z_\ell]=\{0\}$, so $\vartheta_0=0$ and $\theta(y_c) =\vartheta_1 y_c\in \ZC(R) y_c$.
\end{proof}

Our next task is to study how the derivation $\theta$ acts on the generators $x_1, \ldots, x_N$ of $R$. We refer to $\theta$ as a {\em homogeneous derivation}. The terminology will be justified in the following subsection.

\subsection{Homogeneous derivations}\label{SS:homogeneous}

In this subsection, we study the derivation $\theta$ of $R$ appearing in the decomposition~\eqref{E:D-as-ad-plus-central}. By Corollary~\ref{C:theta:action:c}, its action on the homogeneous elements $y_i$, with $i\in [1,N]$, is of the form 
\[
\theta(y_k)=\begin{cases}
\omega_k y_k & \text{if $y_k\notin\ZC(R)$,}\\
\omega_k & \text{if $y_k\in\ZC(R)$,}
\end{cases}
\]
for some $\omega_1, \ldots, \omega_N\in \ZC(R)$. Our aim is to describe how $\theta$ acts on the generators $x_1, \ldots, x_N$ of $R$.

Recall that the QNA $R=R_k[x_{k+1}; \sigma_{k+1}, \delta_{k+1}]\cdots [x_N; \sigma_N, \delta_N]$ is equipped with the action of the maximal torus $\mathcal{H}=(\k^*)^n$ by $\k$-automorphisms, where $n$ is the rank of $R$ and the character group $X(\mathcal{H})$ is isomorphic to $Q:=\Z^n$
(see \cite[Chap. II.2]{bg}). For all homogeneous $x\in R$, there exists a unique {\em weight} $\underline\alpha =(\alpha_1, \ldots, \alpha_n)\in Q$ such that
$$\underline h\cdot x=h_1^{\alpha_1}\ldots h_n^{\alpha_n}x,$$
for all 
$\underline h=(h_1, \ldots, h_n)\in \mathcal{H}.$ 
Denote the weight of $x$ by $\wt(x)=\underline \alpha\in Q.$ Set $$\underline{\beta_i}:=\wt(x_i) \ \text{and} \ \underline{\gamma_i}:=\wt(y_i)$$ for all $i\in [1, N],$ and let $Q_k:=\langle \underline{\beta_1}, \ldots, \underline{\beta_k}\rangle$ be the subgroup of $Q$ generated by $\underline{\beta_1}, \ldots, \underline{\beta_k}$. The subgroup 
$Q_k$ is a free abelian group of rank at most $n$. 

Suppose that $\delta_k\neq 0$. Then, there exists $a\in R_{k-1}$ homogeneous such that 
$$x_ka=\sigma_k(a)x_k+\delta_k(a),$$
with $0\neq\delta_k(a)\in R_{k-1}$ and $\sigma_k(a)\in \k^*a.$ In particular, $\wt(a), \wt(\delta_k(a))\in Q_{k-1}$. Since $x_ka=\sigma_k(a)x_k+\delta_k(a)$ is homogeneous,  we deduce that $\delta_k(a)\in R_{k-1}$ is homogeneous with 
$$\wt(\delta_k(a))=\underline{\beta_k}+\wt(a)\in Q_{k-1}.$$ Thus, 
$$\underline{\beta_k}=\wt(\delta_k(a))-\wt(a)\in Q_{k-1}.$$ It follows that 
$$Q_k=Q_{k-1}.$$

Since we are assuming that $\mathcal{H}$ is maximal, that is, $n=|\{k\mid \delta_k=0\}|$, it follows that, if $\delta_k= 0$, then $Q_k$ has rank equal to one more than the rank of $Q_{k-1}$. The above discussion together with \cite[Theorem 5.5]{yak} shows the following.

\begin{lem}\label{b9}
\begin{enumerate}[(a)]
\item If $\delta_k\neq 0,$ then $Q_k=Q_{k-1}$.
\item If $\delta_k=0,$ then $Q_k=Q_{k-1}\oplus \Z\underline{\beta_k}$.
\end{enumerate}
\end{lem}

We are now ready to describe the action of the homogeneous derivation $\theta$ on the generators $x_i$ (and on all homogeneous elements).

\begin{pro}
\label{centralderivation1}
Let $\theta$ be a $\k$-derivation of $R$. Assume that: 
\begin{enumerate}[(i)]
\item None of the generators $x_i$, with $i\in [1,N]$, is central in $R$. 
\item There are $\omega_1, \ldots, \omega_N\in \ZC(R)$ such that $
\theta(y_k)=\begin{cases}
\omega_k y_k & \text{if $y_k\notin\ZC(R)$;}\\
\omega_k & \text{if $y_k\in\ZC(R)$.}
\end{cases}
$
\end{enumerate}
Then, there exists an abelian group homomorphism $\eta:Q\rightarrow \ZC(R)$
such that, for every homogeneous element $a\in R$, 
\[\theta(a)=\eta(\wt(a))a.\]
\end{pro}

\begin{proof}
It is enough to prove the result by establishing that 
$\theta(x_i)=\eta(\underline{\beta_i})x_i$, for all $i\in [1,N]$.
We do this by proving  that there is an abelian group homomorphism  $\eta_k:Q_k\rightarrow \ZC(R)$
such that $\theta(x_i)=\eta_k(\underline{\beta_i})x_i$ for all $i\in [1,k].$ We proceed by an induction on $k$.  

If $k=1$, then $p(1)=-\infty$, and so $x_1=y_1$. Since $x_1 \notin \ZC(R)$, we get that $\theta (x_1)=\omega_1 x_1$. The result follows by setting  $\eta_1(\underline{\beta_1})=\omega_1$.

Suppose now that the result is true for $k-1$, with $k\geq 2$. Then,
$\theta(a)=\eta_{k-1}(\wt(a))a,$ for all homogeneous elements $a\in R_{k-1}.$ The rest follows in cases.

\underline{Case 1}. $\delta_k=0$. From Lemma~\ref{b9}, $Q_k=Q_{k-1}\oplus \Z\underline{\beta_k}$. Since 
$\delta_k=0,$ we have that $p(k)=-\infty$, hence $y_k=x_k$. As above, our assumptions force $y_k\notin\ZC(R)$ and so $\theta(x_k)=\theta(y_k)=\omega_ky_k=\omega_k x_k$. Thus, we extend 
$\eta_{k-1}$ to $\eta_k$ by setting $\eta_k(\underline{\beta_k}):=\omega_k$.

\underline{Case 2}. $\delta_k\neq 0$ and $y_k\notin\ZC(R)$. Then $Q_k=Q_{k-1}$, by Lemma~\ref{b9}. In this case, for ease of notation, set $\eta:=\eta_{k}=\eta_{k-1}.$ Since $\delta_k\neq 0,$ we have that $p(k)\neq -\infty.$ Set $i:=p(k)\in [1, k-1].$ From \eqref{5T}, we deduce that $y_k=y_ix_k-c_k,$ where $c_k\in R_{k-1}\setminus  y_iR_{k-1} $  (see Remark \ref{T25}). Note that $c_k$ is homogeneous of weight $\wt(x_k)+\wt(y_i)=\underline{\beta_k}+\underline{\gamma_i}$. Apply $\theta$ to $y_k=y_ix_k-c_k$ to obtain
\begin{align*}
\omega_k(y_ix_k-c_k)= \omega_ky_k&=\theta(y_k)=\eta (\underline{\gamma_i})y_ix_k+y_i\theta(x_k)-\eta(\underline{\beta_k}+\underline{\gamma_i})c_k\\
&=\eta(\underline{\gamma_i})y_ix_k-\eta(\underline{\gamma_i})c_k+y_i\theta(x_k)-\eta(\underline{\beta_k})c_k,
\end{align*}
where $\omega_k\in \ZC(R)$. Rearranging terms, we get
\begin{equation}
\label{eq-theta-x_k}
y_i (\omega_k x_k -\eta(\underline{\gamma_i})x_k -\theta(x_k))= (\omega_k -\eta(\underline{\gamma_i}) -\eta(\underline{\beta_k}))c_k.
\end{equation}
Set $z:= \omega_k -\eta(\underline{\gamma_i}) -\eta(\underline{\beta_k}) \in \ZC(R)$, so that $zc_k \in y_iR$.

Since $s(i)=k>k-1$, $y_i$ is a homogeneous prime element of $R_{k-1}$, by \cite[Theorem 4.3]{yak}. From~\eqref{eq-theta-x_k} we have $c_kz=zc_k=y_i r$, for some $r\in R$. Now, expanding $z$ and $r$ in the iterated Ore extension $R_{k-1}[x_k;\sigma_k,\delta_k]\dots [x_N;\sigma_N,\delta_N]$, equating coefficients (in $R_{k-1}$) of equal monomials in $x_k, \ldots, x_N$ coming from the equality $c_k z=y_i r$, using the fact that $y_i$ is a prime element of $R_{k-1}$ and that $c_k \in R_{k-1}\setminus  y_iR_{k-1}$, we can deduce that $z\in y_i R$.

Hence, observing that $s(i)=k < +\infty$, it follows from Proposition \ref{b5} that $z \in  y_iR \cap \NS(R)=\{0\}$. In other words, we have $\omega_k = \eta (\underline{\gamma_i}) +\eta(\underline{\beta_k})$, and we deduce from~\eqref{eq-theta-x_k} that
\[\theta (x_k) = \omega_kx_k-\eta(\underline{\gamma_i}) x_k= \eta(\underline{\beta_k})x_k,\]
as desired.

\underline{Case 3}. $\delta_k\neq 0$ and $y_k\in \ZC(R)$. Then it follows from Lemma~\ref{b9} that $Q_k=Q_{k-1}$. As before, set $\eta:=\eta_{k}=\eta_{k-1}.$ Similarly to the previous case, we have that
 $y_k=y_ix_k-c_k,$ where $c_k\in R_{k-1}\setminus y_iR_{k-1}$, $i=p(k)\in [1, k-1]$, and  $c_k$ is homogeneous of weight $\wt(x_k)+\wt(y_i)=\underline{\beta_k}+\underline{\gamma_i}$. Since $y_k\in \ZC(R)$, we know that  $\theta(y_k)=\omega_k\in\ZC(R)$. Write $\omega_k=\mu_k+\lambda_ky_k$, where $\lambda_k, \mu_k\in \ZC(R)$ with $\deg_{y_k}(\mu_k)=0$. Just as in the previous case, apply $\theta$ to $y_k=y_ix_k-c_k$ to obtain 
$$
\mu_k +y_i (\lambda_k x_k -\eta(\underline{\gamma_i})x_k -\theta(x_k))= (\lambda_k -\eta(\underline{\gamma_i}) -\eta(\underline{\beta_k}))c_k.
$$
Again, we set $z:= \lambda_k -\eta(\underline{\gamma_i}) -\eta(\underline{\beta_k}) \in \ZC(R)$, so that 
$$
\mu_k +y_i (\lambda_k x_k -\eta(\underline{\gamma_i})x_k -\theta(x_k))= zc_k =z(y_ix_k-y_k).
$$
Rearranging terms leads to 
$$
y_i \left( \lambda_k x_k -\eta(\underline{\gamma_i})x_k -\theta(x_k)-zx_k  \right) =  -zy_k -\mu_k \in \ZC(R) \cap y_iR.
$$
Since $s(i)=k <+\infty$, we deduce from Proposition \ref{b5} that 
  \begin{equation}
\label{eq-theta-x_k-2}
y_i \left( \lambda_k x_k -\eta(\underline{\gamma_i})x_k -\theta(x_k)-zx_k  \right) =  -zy_k -\mu_k =0.
\end{equation}
The last equation takes place in the polynomial algebra $\NS(R)=\k_{\q'}[y_j \mid s(j)=+\infty]$ and, by assumption, we have $\deg_{y_k}(\mu_k)=0$. This forces $\mu_k=z=0$ and so we deduce from \eqref{eq-theta-x_k-2} that 
$$\theta (x_k) = \lambda_kx_k-\eta(\underline{\gamma_i}) x_k= \eta(\underline{\beta_k})x_k,$$
as desired.
\end{proof}

Since $x_k$ and $y_k$ are homogeneous elements of $R$, for each $k\in [1,N]$, we have the following immediate corollary.   
\begin{cor}
\label{theta(y_k)inZ(R)y_k}
Under our running assumptions,
$\theta(x_k)\in \ZC(R)x_k$ and $\theta(y_k)\in \ZC(R)y_k$, for all $k\in [1, N].$ 
\end{cor}

We are now in position to describe the first Hochschild cohomology group of $R$. Recall that $\hh(R)=\der(R)/\innder(R)$.

\begin{thm}
\label{thm-Der-QNA}
Let $R=\k[x_1][x_2;\sigma_2,\delta_2]\cdots
[x_N;\sigma_N,\delta_N]$ be a uniparameter QNA of rank $n$. Assume that: 
\begin{enumerate}[(i)]
\item None of the generators $x_i$, with $i\in [1,N]$, is central in $R$;
\item Hypothesis~\ref{hyp} holds.
\end{enumerate}
Then every derivation $D$ of $R$ can be uniquely written as $D=\ad_x + \theta_{\eta}$, where $x \in R$ and $\theta _{\eta}$ is the homogeneous derivation of $R$ associated to the abelian group homomorphism $\eta: Q\rightarrow \ZC(R)$ defined by $\theta_\eta (a) = \eta (\wt(a))a$, for any homogeneous element $a \in R$. 
\end{thm}
\begin{proof}
The existence part follows from Lemma~\ref{L:inner:x-in-R}, Corollary~\ref{cor-inner}, Corollary~\ref{C:theta:action:c} and Proposition~\ref{centralderivation1}. The unicity part follows from the unicity of the decomposition of a derivation of a quantum torus as the sum of an inner derivation and a central derivation by \cite[Corollary 2.3]{op} since we can (uniquely) extend any derivation of $R$ to a derivation of the quantum torus $R[y_1^{-1}, \dots y_N^{-1}]=\mathcal{T}_{\bf{q}}$ (see \eqref{R-torus}). 
\end{proof}

\begin{cor}
\label{cor-HH1-QNA}
$\hh(R)$ is a free $\ZC(R)$-module of rank $\rk(Q)=\rk(R)=n$. 
\end{cor}

The following examples (in fact, non-examples) address the reasonability of the assumptions in Theorem~\ref{thm-Der-QNA}.

The example below shows that the conclusion of our main results can fail if the Hypothesis~\ref{hyp} is removed.

\begin{exa}[{\cite[Corollaire 1.3.3]{ac}}]
Let $R=\mathcal A_\q$, the quantum affine space associated with the matrix $\q=
\begin{psmallmatrix}
1& q & q\\
q^{-1} &1 & q\\
q^{-1} &q^{-1} &1 \\ 
\end{psmallmatrix}
$, where $q\in\k^*$ is not a root of unity. Then $R$ is a uniparameter QNA of rank $3$ with $\ZC(\mathcal{T}_\q)=\k[z]$ for $z=x_1x_2^{-1}x_3$ and $\ZC(R)=\k$, and none of the canonical generators is central in $R$. However, as shown in~\cite[Corollaire 1.3.3]{ac}, $\hh(R)$ is a free $\ZC(R)$-module of rank $4$. Besides the derivations described in the conclusion of Theorem~\ref{thm-Der-QNA}, which induce a $3$-dimensional subspace of $\hh(R)$, there is an additional derivation $\delta$ such that $\delta(x_1)=\delta(x_3)=0$ and $\delta(x_2)=x_1 x_3$. Thus our main results do rely on Hypothesis~\ref{hyp}.
\end{exa}

The following example shows that the first Hochschild cohomology group of an iterated Ore extension, as in~\eqref{T3}, can even fail to be free as a module over the center of the Ore extension, e.g.\ if working with roots of unity.

\begin{exa}[{\cite[Theorem 12]{GG14}}]\label{ex:qweylalg:1}
Let $R=\k[x][y;\sigma, \delta]$ be the \textit{quantum Weyl algebra} defined by the relation $yx=qxy+1$, so $\sigma(x)=qx$ and $\delta(x)=1$. If $q$ is a primitive $\ell$-th root of unity, for some $\ell>1$, then $R$ is not a QNA, although the only condition that it fails is that $q_2=q^{-1}$ is a root of unity, contradicting part (c) of Definition~\ref{dcgl} (we're using the notation $q_k$ introduced in that definition). In this case (see \cite[Theorem 12]{GG14}), $\ZC(R)=\k[x^\ell, y^\ell]$ and $\hh(R)$ is not free over $\ZC(R)$. 
\end{exa}

The next example, related to the Lie algebra $\mathfrak{sl}_2$, shows that the existence of central variables can affect the conclusion of Theorem~\ref{thm-Der-QNA}.

\begin{exa}\label{ex:uqsl2}
Let $R=U_q^+(\mathfrak{sl}_2)=\k[x]$ (see below for the definition of $U_q^+(\mathfrak{g})$). Then $R$ is trivially a uniparameter QNA of rank $1$ satisfying all the hypotheses of Theorem~\ref{thm-Der-QNA} except that the generator $x$ is central. Although the conclusion of Corollary~\ref{cor-HH1-QNA} holds for $R$, the conclusion of Theorem~\ref{thm-Der-QNA} fails, as the usual derivative $\frac{d}{dx}$ does not send $x$ to a (central) multiple of $x$.
\end{exa}

If $R$ is a QNA having central variables $X_1=x_{j_1}, \ldots, X_m=x_{j_m}$, such that none appears in the expressions for the skew-derivations $\delta_i$ in~\eqref{T3}, then the central variables can be added at the end, so that we get a presentation of the form $R=A[X_1, \ldots, X_m]$, with $A$ a QNA with no central variables. Thus, assuming that Theorem~\ref{thm-Der-QNA} can be applied to $A$, we can obtain $\der(R)$ and $\hh(R)$ by using Theorem~\ref{T:generic} in combination with Theorem~\ref{thm-Der-QNA}.

Our final example combines features of Example~\ref{ex:uqsl2} and Example~\ref{ex:qweylalg:1} and shows that, in case there is a central variable which occurs in the expression of a skew-derivation, then Theorem~\ref{T:generic} cannot be applied.

\begin{exa}\label{ex:qweylalg:2}
Let $R=\k[x][y][z;\sigma, \delta]$ be the QNA defined by the relation $zy=qyz+x$, with $q\in\k^*$ not a root of unity and $x$ central in $R$. Thus, $\sigma(x)=x$, $\sigma(y)=qy$, $\delta(x)=0$ and $\delta(y)=x$. Let $\overline R$ be the localization of $R$ at the Ore set of powers of $x$. Then $\overline R=A[x^{\pm 1}]$, a (Laurent) polynomial ring over the quantum Weyl algebra $A$, generated by $yx^{-1}$ and $z$. Applying Theorem~\ref{T:generic} to $\overline R$ and then restricting to $R$, it can be shown that
\begin{equation*}
\der(R)= \innder(R)\oplus \ZC(R)\delta_1\oplus \ZC(R) \delta_2,
\end{equation*}
where $\delta_1(x)=0$, $\delta_1(y)=y$, $\delta_1(z)=-z$, and $\delta_2(x)=x$, $\delta_2(y)=y$, $\delta_2(z)=0$. In particular, there is no derivation of $R$ that annihilates both $y$ and $z$ and sends $x$ to $1$.
\end{exa}


\section{Application to quantized enveloping algebras}
\label{u+}
Let $\mathfrak{g}$ denote a finite-dimensional complex simple Lie algebra of rank $n$ and $U_q(\mathfrak{g})$ be the corresponding quantized enveloping algebra over a base field $\k$ of arbitrary characteristic and deformation parameter $q\in \k^*$ that is not a root of unity. We denote by $E_i,F_i,K_i^{\pm 1}$, with $i\in [1,n]$, the standard generators of $U_q(\mathfrak{g})$. 
Let $U_q^+(\mathfrak{g})$ be the positive part of $U_q(\mathfrak{g})$, that is, the subalgebra of  $U_q(\mathfrak{g})$  generated by the $E_i$ with $i\in [1,n]$. This is the algebra generated by $E_1, \dots ,E_n$, subject to the quantum Serre relations. 

It is well known that  $U_q^+(\mathfrak{g})$ is a uniparameter QNA of rank $n$ and of length equal to the number $N$ of positive roots of $\mathfrak{g}$, see for instance \cite[Chap. 9]{yak1}.

In order to apply the results of the previous section to the uniparameter QNA $U_q^+(\mathfrak{g})$, we need some control over the center of this algebra. We will see that Hypothesis~\ref{hyp} holds for all simple $\mathfrak{g}$ and that the hypothesis on the generators $x_i$ not being central holds as well, as long as $n=\rk(\mathfrak{g})\neq 1$. 

We can immediately tackle the cases for which $\NS(U_q^+(\mathfrak{g}))=\ZC(U_q^+(\mathfrak{g}))$, which ensures that the Hypothesis~\ref{hyp} holds (see Remark~\ref{R:rem:on:hyp}). In fact, it follows from \cite[Remarque 2.2]{caldero1} that normal  and central elements coincide if and only if the longest element $w_0$ of the Weyl group satisfies $w_0=-1$, that is, if and only if $\mathfrak{g}$ is of type $A_1$, $B_n  ~(n\geq 2)$, $C_n ~(n\geq 3)$, $D_n ~(n\geq 4 ~\text{even})$, $G_2$, $F_4$, $E_7$ or $E_8$. The remaining cases still satisfy the Hypothesis~\ref{hyp}, as we will show that, for any simple $\mathfrak g$, there is a partition $(Z_1, \ldots, Z_\ell)$ of $\mathfrak{s}_{+\infty}$ such that, up to nonzero scalar factors, $z_i=\prod_{k\in Z_i}y_k$, for $1\leq i\leq \ell$. The latter implies the Hypothesis~\ref{hyp}, by Remark~\ref{R:rem:on:hyp}.

\begin{thm}\label{T:hyp:star:holds:Uq+g}
Let $\mathfrak{g}$ be a finite-dimensional complex simple Lie algebra of rank $n\geq 1$ and $R=U_q^+(\mathfrak{g})$. There is a partition $(Z_1, \ldots, Z_\ell)$ of $\mathfrak{s}_{+\infty}$ such that, taking $z_i=\prod_{k\in Z_i}y_k$, for $1\leq i\leq \ell$, the following hold:
\begin{enumerate}[(a)]
\item $\ZC(\mathcal{T}_{\q})=\k[z_1^{\pm 1}, \ldots, z_\ell^{\pm 1}]$;
\item $\ZC(\mathcal{A}_{\q})=\k[z_1, \ldots, z_\ell]$.
\end{enumerate}
In particular, the Hypothesis~\ref{hyp} holds for $U_q^+(\mathfrak{g})$.
\end{thm}
\begin{proof}
We will use \cite{caldero}, but to avoid confusion with the generators of $R$ as an Ore extension, in this proof we denote by $\Delta_i$ the elements denoted in \cite{caldero} by $x_{s(\varpi_i)}$, for $1\leq i\leq n$, where $\varpi_1,\ldots, \varpi_n$ are the fundamental weights. It will be shown in Theorem~\ref{T:Caldero:equals:GY} that, up to a nonzero scalar factor and a permutation of the indices, the elements $\Delta_1, \ldots, \Delta_n$ are precisely the prime homogeneous elements $\{y_i\}_{s(i)=+\infty}$ of $R$, say $\Delta_k=y_{i_k}$, for $1\leq k\leq n$ such that $\mathfrak{s}_{+\infty}=\{i_k\mid 1\leq k\leq n\}$. 

Thus, by Lemme~2.3, Proposition~3.2 and Th\'eor\`eme~3.2 in \cite{caldero}, it follows that there is a partition $(Z_1, \ldots, Z_\ell)$ of $\mathfrak{s}_{+\infty}$ such that, defining $z_i=\prod_{k\in Z_i}y_k$, for $1\leq i\leq \ell$, we get $\ZC(U_q^+(\mathfrak{g}))=\k[z_1, \ldots, z_\ell]$. (For example, in type $A_n$ we have $\ell=\frac{n+1}{2}$ and, if $n$ is odd, we can take 
$Z_k=\{i_k, i_{n+1-k}\}$, $z_k=y_{i_k}y_{i_{n+1-k}}$, for $1\leq k<\frac{n+1}{2}$, and $Z_{\frac{n+1}{2}}=\left\{i_{\frac{n+1}{2}}\right\}$, $z_{\frac{n+1}{2}}=y_{i_{\frac{n+1}{2}}}$.)

By~\eqref{E:chain:center:R:Aq:Tq}, we have that $\ZC(\mathcal{A}_{\q})=\ZC(R)=\k[z_1, \ldots, z_\ell]$. Let us consider the center of the quantum torus $\mathcal{T}_{\q}$. Clearly, $\ZC(\mathcal{T}_{\q})\supseteq \k[z_1^{\pm 1}, \ldots, z_\ell^{\pm 1}]$. For the reverse inclusion, let $z\in\ZC(\mathcal{T}_{\q})$. By \cite[Proposition 2.11]{yak2}, we have $\ZC(\mathcal{T}_\q)\subseteq \NS(R)E_{+\infty}^{-1}$, so there is a monomial $u$ in the variables $Y_{+\infty}$ such that $zu\in\NS(R)$. As the supports of the central elements $z_1, \ldots, z_\ell$ cover $\mathfrak{s}_{+\infty}$, we can assume that $u$ is a monomial in the $z_i$. Thus, $zu\in\NS(R)\cap\ZC(\mathcal{T}_\q)=\ZC(\mathcal{A}_{\q})$, by~\eqref{E:chain:center:R:Aq:Tq}. So $z\in\ZC(\mathcal{A}_{\q})u^{-1}\subseteq\k[z_1^{\pm 1}, \ldots, z_\ell^{\pm 1}]$.

Whence, starting with $\ZC(\mathcal{T}_{\q})=\k[z_1^{\pm 1}, \ldots, z_\ell^{\pm 1}]$ with the $z_i$ as above, we see that the Hypothesis~\ref{hyp} holds, using Remark~\ref{R:rem:on:hyp}.
\end{proof}

In the case where $\mathfrak{g}$ is of type $A_1$, we have that $U_q^+(\mathfrak{g}) = \k [E_1]$ and so the first assumption of Theorem~\ref{thm-Der-QNA} is not satisfied in that case. However, in all other cases, none of the root vectors are central (or, equivalently, each simple root appears at least twice in the support of any reduced decomposition of the longest element $w_0$ of the Weyl group $W$, see for instance \cite[Section 6]{Meriaux}), whence all assumptions of Theorem~\ref{thm-Der-QNA} are satisfied for $U_q^+(\mathfrak{g})$ when $\mathfrak{g}$ is not of rank $1$.

We can give an easy description of the homogeneous derivations of $U_q^+(\mathfrak{g})$: $\wt(E_1)$, \ldots, $\wt(E_n)$ are free generators of $Q$ so, for $i\in [1,n]$, if $\alpha^*_i:Q\rightarrow \ZC(R)$ is the group homomorphism that sends $\wt(E_j)$ to $\delta_{ij}1\in \ZC(R)$, then $\Hom(Q, \ZC(R))$ is the free $\ZC(R)$-module on the basis $\alpha^*_1$, \ldots, $\alpha^*_n$. Set $D_i=\theta_{\alpha^*_i}$. Then $D_i (E_j)=\delta_{ij}E_j$ and $\hh(U_q^+(\mathfrak{g}))$ is a free $\ZC(U_q^+(\mathfrak{g}))$-module of rank $n=\rk(\mathfrak{g})$ with basis $\{\overline{D}_1$, ..., $\overline{D}_n\}$.

As a consequence of the above discussion, we deduce from Theorem \ref{thm-Der-QNA} and Corollary \ref{cor-HH1-QNA} the following description of the first cohomology group of $U_q^+(\mathfrak{g})$.

\begin{thm}
\label{finalder}
Let $R=U_q^+(\mathfrak{g})$, where $\mathfrak{g}$ is a finite-dimensional complex simple Lie algebra of rank $n\geq 2$. Then:
\begin{enumerate}[(a)]
\item $\der(R)=\innder(R)\oplus \displaystyle{\bigoplus_{k=1}^n} \ZC(R)D_k$.
\item $\hh(R)$ is a free $\ZC(R)$-module of rank $n$ with basis $\{\overline{D}_1, \ldots, \overline{D}_n\}$.
\end{enumerate}
\end{thm}

We note that similar results have been obtained for other QNAs not satisfying our hypotheses, e.g. quantum matrices~\cite{sltl}. However, the methods developed in the present paper do not yet allow to prove this result in full generality.

\appendix

\section{Proof of Theorem~\ref{b4} without the symmetry assumption}\label{App}

We will provide a proof of Theorem~\ref{b4} without using the assumption that the QNA $R$ is symmetric. We need a few preliminary results.

\begin{lem}
\label{b8}
Let $j\in [1,N]$ and suppose that $rw \in y_jR$, with $w\in R$ and $r \in R_j$, where $R_j$ is as defined in \eqref{R_k}. If $r\not\in y_jR_j$ then $w\in y_jR$.  
\end{lem}
\begin{proof}
Write $R=R_j[x_{j+1}; \sigma_{j+1}, \delta_{j+1}]\cdots [x_N; \sigma_N, \delta_N]$. 
Since $rw\in y_jR$, we have that $rw=y_js$, for some $s\in R$.
Now $w,s\in R$  can be written as:
\begin{equation*}
w=\sum_{{\underline f}\in (\Z_{\geq 0})^{N-j}}w_{\underline f} x_{j+1}^{f_{j+1}}\cdots x_N^{f_N} \ \ \text{and} \ \ \ s=\sum_{{\underline g}\in (\Z_{\geq 0})^{N-j}}s_{\underline g}x_{j+1}^{g_{j+1}}\cdots x_N^{g_N}, 
\end{equation*}
where ${\underline f}=(f_{j+1}, \ldots, f_N)$, ${\underline g}=(g_{j+1}, \ldots, g_N)$ and $w_{\underline f}, s_{\underline g}\in R_j$.

Therefore, $rw=y_js$ implies that
\begin{equation*}
\sum_{{\underline f}\in (\Z_{\geq 0})^{N-j}} rw_{\underline f}x_{j+1}^{f_{j+1}}\cdots x_N^{f_N} =\sum_{{\underline g}\in (\Z_{\geq 0})^{N-j}}y_js_{\underline g}x_{j+1}^{g_{j+1}}\cdots x_N^{g_N}. 
\end{equation*}
Consequently, $rw_{\underline f}=y_js_{\underline f}$ for all ${\underline f}\in (\Z_{\geq 0})^{N-j}$. Note that $y_j$ is a prime element of $R_j$ since $s(j)>j$, hence $y_jR_j$ is a completely prime ideal of $R_j$. Since we assume that $r\not\in y_jR_j$, we have that $w_{\underline f}\in y_jR_j$, for all ${\underline f}\in (\Z_{\geq 0})^{N-j}$.
Therefore, $w_{\underline f}=y_ju_{\underline f}$, for some $u_{\underline f}\in R_j.$ It follows that there exists $d\in R$ such that $w=y_jd\in y_jR.$ 
\end{proof}

\begin{pro}
\label{L}
For all $i, j\in [1, N]$ with $i\neq j$, we have that
 $y_iR\cap y_jR=y_iy_jR$.
\end{pro}
\begin{proof}
Without loss of generality, we assume that $i<j$. Clearly, $y_iR\cap y_jR\supseteq y_iy_jR$ as $y_i$ and $y_j$ quasi-commute. For the reverse inclusion, let 
$x\in y_iR\cap y_jR$. Then, there exists $w\in R$ such that 
$x= y_iw\in y_jR$.  
 Since $y_iw\in y_jR$, with $y_i\not\in y_jR_j$ (see Lemma~\ref{dd}), it follows from Lemma~\ref{b8} that $w\in y_jR$. This implies that 
 $w=y_jd$, for some $d\in R$. Consequently,  $x= y_iw=y_iy_jd\in y_iy_jR$. This establishes the reverse inclusion. 
\end{proof}

We are now ready to give the general proof of Theorem~\ref{b4}.

\begin{proof}[Proof of Theorem~\ref{b4}]
We begin with some temporary notation. Given a subset $K\subseteq [1,N]$, let 
\begin{equation*}
\left(\Z_{\geq 0}\right)^K:=\{(f_1, \ldots, f_N)\in \left(\Z_{\geq 0}\right)^N\mid f_i=0\ \text{for all } i\notin K\}, 
\end{equation*}
and for ${\underline f}=(f_1, \ldots, f_N)\in\left(\Z_{\geq 0}\right)^K$, let ${\underline y}^{\underline f}:=y_1^{f_1}\cdots y_N^{f_N}$. We denote the canonical basis of the free abelian group $\Z^N$ by $(\epsilon_1, \ldots, \epsilon_N)$, so that $\epsilon_k$ is the element of $\Z^N$ with all of its coordinates equal to $0$ except for the $k$-th coordinate, which is $1$.

The inclusion $RE_{I\cap J}^{-1}\subseteq RE_I^{-1} \cap RE_J^{-1}$ is clear. For the reverse inclusion, let $x\in RE_I^{-1} \cap RE_J^{-1}$. As $x\in RE_I^{-1}$, there is ${\underline f}\in\left(\Z_{\geq 0}\right)^I$, minimal with respect to the lexicographic order, so that ${\underline y}^{\underline f} x\in R$. Suppose, by way of contradiction, that $x\notin RE_{I\cap J}^{-1}$. Then, ${\underline f}\notin\left(\Z_{\geq 0}\right)^{I\cap J}$, so there is $i\notin I\cap J$ with $f_i>0$. As ${\underline f}\in\left(\Z_{\geq 0}\right)^I$, it follows that $i\in I\setminus J$.

Let ${\underline f'}={\underline f}-\epsilon_i\in\left(\Z_{\geq 0}\right)^I$ and set $x'=\underline y^{\underline f'}x$. By the minimality of ${\underline f}$, $x'\notin R$. However, since the $y_k$ pairwise quasi-commute, we have $y_ix'\in R$ and $x'\in RE_J^{-1}$ because $x\in RE_J^{-1}$.

Repeating the argument above with $x'$, we deduce that there is ${\underline g}\in\left(\Z_{\geq 0}\right)^J$, minimal with respect to the lexicographic order, so that ${\underline y}^{\underline g} x'\in R$. As $x'\notin R$, there is some $j$ such that $g_j>0$; in particular, $j\in J$ and thus $i\neq j$. Then, using Proposition~\ref{L} and the fact that the $y_k$ pairwise quasi-commute, we get
\begin{equation*}
y_i {\underline y}^{\underline g} x'\in y_i R\cap y_j R=y_i y_j R.
\end{equation*}
Recalling that $R$ is a domain, we deduce that ${\underline y}^{\underline g'} x'\in R$, for ${\underline g'}={\underline g}-\epsilon_j\in\left(\Z_{\geq 0}\right)^J$. This contradicts the minimality of ${\underline g}\in\left(\Z_{\geq 0}\right)^J$ and this contradiction implies that indeed $x\in RE_{I\cap J}^{-1}$.
\end{proof}

\section{The normal elements in $U_q^+(\mathfrak{g})$}\label{AppB}

To have a complete proof that the Hypothesis~\ref{hyp} holds for $U_q^+(\mathfrak{g})$, as stated in Theorem~\ref{T:hyp:star:holds:Uq+g}, we need to show that the prime homogeneous elements $Y_{+\infty}$ are precisely, up to re-ordering and scaling, the elements introduced by Caldero in \cite{caldero, caldero1}. We will sketch this here, following mostly \cite{caldero1} and the notation therein (note that there are some divergences in notation from \cite{caldero} to \cite{caldero1}).

Fix a finite-dimensional complex simple Lie algebra
$\mathfrak{g}$ of rank $n\geq 1$ with $P=\bigoplus_{i=1}^n \Z \varpi_i$ the integral weight lattice and $P^+=\bigoplus_{i=1}^n \Z_{\geq 0}\varpi_i$ the set of dominant integral weights, where $\varpi_1,\ldots, \varpi_n$ are the fundamental weights. Let $w_0$ be the longest element of the Weyl group of $\mathfrak{g}$ and denote by $1+ w_0$ and $1- w_0$ the endomorphisms $1_P+ w_0 1_P$ and $1_P- w_0 1_P$ of $P$, respectively, where $1_P$ is the identity on $P$.

For each $\mu\in P^+$, there is a normal element $e_{s(\mu)}\in\NS(U_q^+(\mathfrak{g}))$ of weight $(1- w_0)(\mu)$ and these elements can be chosen so that $e_{s(\mu+\lambda)}=e_{s(\mu)}e_{s(\lambda)}$, for all $\lambda, \mu\in P^+$ (see~\cite[Section 1.5]{caldero1}). Since the set $\{e_{s(\mu)}\mid \mu\in P^+\}$ is linearly independent, it follows that the elements $\Delta_i:=e_{s(\varpi_i)}$, with $1\leq i\leq n$, generate a commutative polynomial subalgebra of $\NS(U_q^+(\mathfrak{g}))$, of Gelfand--Kirillov dimension $n$.

Let $0\neq y\in \NS(U_q^+(\mathfrak{g}))$. By~\cite[Th\'eor\`eme 2.2]{caldero1}, there exist $k\geq 1$, (distinct) $\mu_1, \ldots, \mu_k\in P^+$ with $\mu_i-\mu_j\in\ker(1+w_0)$ for all $i,j$, and $\lambda_1, \ldots, \lambda_k\in\k^*$ such that
\begin{equation*}
y= \sum_{i=1}^k \lambda_i e_{s(\mu_i)}.
\end{equation*}
If we further assume that $y$ is homogeneous, then we had the additional condition that $\mu_i-\mu_j\in\ker(1-w_0)$ for all $i,j$. As $\ker(1+w_0)\cap \ker(1-w_0)=\{0\}$. It follows that $k=1$ and and $y$ is a scalar multiple of some $e_{s(\mu)}$, with $\mu\in P^+$. So, up to a nonzero scalar factor, the normal homogeneous elements of $U_q^+(\mathfrak{g})$ are exactly the monomials in the $\Delta_i$ with $1\leq i\leq n$. Whence, with the additional assumption that $y$ is prime, we deduce that $y$ is a nonzero scalar multiple of $\Delta_i$, for some $1\leq i\leq n$. As, up to scalars, there are exactly $\rk(R)=\rk(\mathfrak{g})=n$ homogeneous prime elements in $R$, it follows that, up to scalars, $Y_{+\infty}=\{\Delta_1, \ldots, \Delta_n\}$.

\begin{thm}\label{T:Caldero:equals:GY}
Let $R=U_q^+(\mathfrak{g})$ and $\Delta_i=e_{s(\varpi_i)}$ be as above. Then,
up to a nonzero scalar factor and a permutation of the indices, the elements $\Delta_1, \ldots, \Delta_n$ are precisely the prime homogeneous elements $\{y_i\}_{i \in \mathfrak{s}_{+\infty}}$ of $R$. In particular, $\NS(U_q^+(\mathfrak{g}))$ is a commutative polynomial algebra.
\end{thm}

\end{document}